\documentclass[12pt]{amsart}
\usepackage{amsmath, amscd}
\usepackage{amsfonts}
\usepackage{version}
\usepackage{amssymb}
\newcommand{\R}{\mathbb R}

\newcommand{\Z}{\mathbb Z}
\newcommand{\C}{\mathbb C}

\newcommand{\B}{\mathbb B}

\newcommand{\D}{\mathbb D}

\newcommand{\N}{\mathbb N}

\def\Re{{\sf Re}\,}
\def\Im{{\sf Im}\,}

\newtheorem{theorem}{Theorem}[section]
\newtheorem{lemma}[theorem]{Lemma}
\newtheorem{proposition}[theorem]{Proposition}
\newtheorem{corollary}[theorem]{Corollary}

\theoremstyle{definition}
\newtheorem{definition}[theorem]{Definition}

\theoremstyle{remark}
\newtheorem{remark}[theorem]{Remark}
\numberwithin{equation}{section}

\numberwithin{equation}{section}

\def\sideremark#1{\ifvmode\leavevmode\fi\vadjust{
\vbox to0pt{\hbox to 0pt{\hskip\hsize\hskip1em
\vbox{\hsize1.5cm\tiny\raggedright\pretolerance10000
\noindent #1\hfill}\hss}\vbox to8pt{\vfil}\vss}}}

\begin{document}
\title[Non-simply connected Fatou components]{Automorphisms of $\C^k$ with an invariant non-recurrent attracting Fatou component biholomorphic to $\C\times (\C^\ast)^{k-1}$}

\author[F. Bracci]{Filippo Bracci$^{\diamondsuit\star}$}\thanks{$^\diamondsuit$Partially supported by the ERC grant ``HEVO - Holomorphic Evolution Equations'' n. 277691 and the MIUR Excellence Department Project awarded to the  
Department of Mathematics, University of Rome Tor Vergata, CUP E83C18000100006}
\address{F. Bracci: Dipartimento di Matematica\\
Universit\`{a} di Roma \textquotedblleft Tor Vergata\textquotedblright\ \\
Via Della Ricerca Scientifica 1, 00133 \\
Roma, Italy} \email{fbracci@mat.uniroma2.it}
\author[J. Raissy]{Jasmin Raissy$^\spadesuit$}\thanks{$^\spadesuit$Partially supported by the ANR
 project LAMBDA,  ANR-13-BS01-0002 and by
 the FIRB2012 grant ``Differential Geometry and Geometric Function Theory'', RBFR12W1AQ 002.}
\address{J. Raissy: Institut de Math\'ematiques de Toulouse; UMR5219,
Universit\'e de Toulouse; CNRS, UPS IMT, F-31062 Toulouse Cedex 9, France.}\email{jraissy@math.univ-toulouse.fr}
\author[B. Stens\o nes]{Berit Stens\o nes$^{\clubsuit\star}$}
\thanks{$^\clubsuit$Partially supported by the FRIPRO Project n.10445200}
\address{B. Stens\o nes: Department of Mathematics, Norwegian University of Science and Technology, Alfred Getz vei 1, Sentralbygg II 950, Trondheim, Norway}\email{berit.stensones@math.ntnu.no}
\thanks{$^\star$Members  of the 2016-17 CAS project {\sl Several Complex Variables and Complex Dynamics}.}
\subjclass[2010]{Primary 37F50; Secondary 32A30, 39B12}
\keywords{Fatou sets; holomorphic dynamics; hyperbolic distance; Runge domains}

\maketitle
\begin{abstract}
We prove the existence of automorphisms of $\mathbb C^k$, $k\ge 2$, having an invariant, non-recurrent  Fatou component biholomorphic to $\mathbb C \times (\mathbb C^\ast)^{k-1}$ which is attracting, in the sense that all the orbits converge to a fixed point on the boundary of the component. As a corollary, we obtain a Runge copy of $\mathbb C \times (\mathbb C^\ast)^{k-1}$ in $\mathbb C^k$. 
The constructed Fatou component also avoids $k$ analytic discs intersecting transversally at the fixed point. 
\end{abstract}

\tableofcontents

\section*{Introduction}

Let $F$ be a holomorphic endomorphism of $\C^k$, $k\ge 2$. In the study of the dynamics of $F$, that is of the behavior of its iterates, a natural dichotomy is given by the division of the space into the {\sl Fatou set} and the {\sl Julia set}. The Fatou set is the largest open set where the family of iterates is locally normal, that is the set formed by all points having an open neighborhood where the restriction of the iterates of the map forms a normal family. The Julia set is the complement of the Fatou set and is the part of the space where chaotic dynamics happens. A {\sl Fatou component} is a connected component of the Fatou set.

A Fatou component  $\Omega$ for a map $F$ is called {\sl invariant} if $F(\Omega)=\Omega$.

We call an invariant Fatou component $\Omega$ for a map $F$ {\sl attracting} if there exists a point $p\in \overline{\Omega}$ with $\lim_{n\to \infty}F^n(z)=p$ for all $z\in \Omega$. Note that, in particular, $p$ is a fixed point for $F$. If $p\in \Omega$ then  $\Omega$ is called {\sl recurrent}, and it is called {\sl non-recurrent} if $p\in \partial \Omega$.

Every attracting recurrent Fatou component of a holomorphic automorphism $F$ of $\C^k$ is biholomorphic to $\C^k$. In fact it is the global basin of attraction of $F$ at $p$, which is an attracting fixed point, that is all eigenvalues of $dF_p$ have modulus strictly less than $1$ (see \cite{PVW} and \cite{RR}).

As a consequence of the results obtained by T. Ueda in \cite{U0} and of Theorem 6 in \cite{LP} by M. Lyubich and H. Peters, every non-recurrent invariant attracting Fatou component $\Omega$ of a \emph{polynomial} automorphism of $\C^2$ is biholomorphic to $\C^2$. L. Vivas and the third named author in \cite{SV} produced examples of automorphisms of $\C^3$ having attracting non-recurrent Fatou component biholomorphic to $\C^2\times\C^*$.

\smallskip 
The main result of our paper is the following:

\begin{theorem}\label{main}
Let $k\ge 2$. There exist holomorphic automorphisms of $\C^k$  having an invariant, non-recurrent,  attracting Fatou component biholomorphic to $\C\times (\C^*)^{k-1}$. 
\end{theorem}

In particular, this shows that there exist (non polynomial) automorphisms of $\C^2$ having a non-simply connected attracting non-recurrent Fatou component. Our construction also shows that the invariant non-recurrent attracting Fatou component biholomorphic to $\C\times(\C^*)^{k-1}$ avoids $k$ analytic discs which intersect transversally at the fixed point. Moreover as a corollary of Theorem~\ref{main} and \cite[Proposition~5.1]{U0}, we obtain:

\begin{corollary}
Let $k\ge 2$. There exists a biholomorphic image of $\C\times(\C^*)^{k-1}$ in $\C^k$ which is Runge.
\end{corollary}

The existence of an embedding of $\C\times\C^*$ as a Runge domain in $\C^2$ was a long standing open question,  positively settled by our construction. After a preliminary version of this manuscript was circulating,  F. Forstneri\v{c} and E. F. Wold constructed in \cite{FrancErlend} other examples of Runge embeddings of $\C\times \C^\ast$ in $\C^2$ (which do not arise from basins of attraction of automorphisms) using completely different techniques.
 
Notice that, thanks to the results obtained by J. P. Serre in \cite{Serre} (see also \cite[Theorem 2.7.11]{Horm}), every Runge domain $D\subset\C^k$ satisfies $H^q(D)=0$ for all $q\ge k$. Therefore the Fatou component in Theorem \ref{main} has the highest possible admissible non-vanishing cohomological degree.

The proof of Theorem \ref{main} is rather involved and we give an outline in the next section. In the rest of the paper, we will first go through the proof in the case $k=2$, and then show the modifications needed for all dimensions. 

The proof relies on a mixture of known techniques and new tools. We first choose a suitable germ having a local basin of attraction with the proper connectivity and extend it to an automorphism $F$ of $\C^k$. Using more or less standard techniques we extend the local basin to a global basin of attraction $\Omega$ of $F$ and then we define a Fatou coordinate. Next, we exploit a new construction to prove that the Fatou coordinate is in fact a fiber bundle map, allowing us to show that $\Omega$ is biholomorphic to $\C\times (\C^*)^{k-1}$. The final rather subtle point is to show that $\Omega$ is indeed a Fatou component. 
We have to introduce a completely new argument, which is based on P\"oschel's results in \cite{Po} and detailed estimates for the Kobayashi metric on certain domains.

\medskip

\noindent{\bf Acknowledgements.} Part of this paper was written while the first and the third named authors were visiting the Center for Advanced Studies in Oslo for the 2016-17 CAS project Several Complex Variables and Complex Dynamics. They both thank CAS for the support and for the wonderful atmosphere experienced there.

The authors also thank Han Peters for some useful conversations, and the anonymous referee, whose comments and remarks improved the presentation of the original manuscript.

\section{Outline of the proof in dimension 2}

For the sake of simplicity, we give the outline of the proof for $k=2$.
We start with a germ of biholomorphism at the origin of the form
\begin{equation}\label{form-intro}
F_N(z,w)=\left(\lambda z\left(1 - \frac{zw}{2}\right), \overline{\lambda}w \left(1 - \frac{zw}{2}\right)\right),
\end{equation}
where $\lambda\in \C$, $|\lambda|=1$, is not a root of unity and satisfies the Brjuno condition \eqref{eq:brjuno}.
Thanks to a result of  B. J.  Weickert \cite{W1} and F. Forstneri\v{c} \cite{F}, for any large $l\in \N$ there exists an automorphism $F$ of $\C^2$ such that 
\begin{equation}\label{Eq-motiv_2}
F(z,w)-F_N(z,w)=O(\|(z,w)\|^l).
\end{equation}
These kind of maps are a particular case of the so-called {\em one-resonant} germs. Recall that a germ of biholomorphism $F$ of $\C^2$ at the origin is called {\em one-resonant} if, 
denoting by $\lambda_1, \lambda_2$ the eigenvalues of its linear part, there exists a fixed multi-index $P = (p_1, p_2)\in\N^2$ with $p_1+p_2\ge 2$ such that all the resonances $\lambda_j - \lambda_1^{m_1}\lambda_2^{m_2}=0$, for $j=1,2$, are precisely of the form $\lambda_j = \lambda_j\cdot\lambda_1^{kp_1}\lambda_2^{kp_2}$ for some $k\ge 1$.

The local dynamics of one-resonant germs has been studied by the first named author with D. Zaitsev in \cite{BZ} (see also \cite{BRZ}). 

Let
\[
B:=\{(z,w)\in \C^2: zw\in S, |z|<|zw|^\beta, |w|<|zw|^\beta\},
\]
where $\beta\in (0,\frac{1}{2})$ and $S$ is a small sector in $\C$ with vertex at $0$ around the positive real axis.
In \cite{BZ} (see also Theorem \ref{Thm:BZ}) it has been proved that for sufficiently large $l$ the domain $B$ is forward invariant under $F$, the origin is on the boundary of $B$ and  $\lim_{n\to \infty}F^n(p)=0$ for all $p\in B$. 
Moreover, setting $x=zw,y=w$ (which are coordinates on $B$) the domain becomes $\{(x,y)\in \C\times\C^*: x\in S, |x|^{1-{\beta}}<|y|<|x|^{\beta}\}$. Hence $B$ is doubly connected.

Now let 
\[
\Omega:=\cup_{n\in \N}F^{-n}(B).
\] 
The domain $\Omega$ is connected but not simply connected. 

For a point $(z,w)\in \C^2$, let $(z_n,w_n):=F^n(z,w)$.
In Theorem \ref{characterized Omega} we show that
\[
\Omega=\{(z,w)\in \C^2\setminus\{(0,0)\}: \lim_{n\to \infty}\|(z_n,w_n)\|=0, \quad |z_n|\sim |w_n|\},
\]
and moreover, if $(z,w)\in \Omega$ then $|z_n|\sim |w_n|\sim \frac{1}{\sqrt{n}}$.
\smallskip

Having a characterization of the behavior of the orbits of a map on a completely invariant domain is however in general not enough to state that such a domain is the whole Fatou component, as this trivial example illustrates:  the automorphism $(z,w)\mapsto (\frac{z}{2}, \frac{w}{2})$ has the completely invariant domain $\C^\ast\times \C^\ast$ which is not a Fatou component but $|z_n|\sim |w_n|$. 

In order to prove that $\Omega$ coincides with the Fatou component $V$ containing it, we exploit the condition that $\lambda$ is also Brjuno (see Section \ref{FGB} for details). In this case there exist two $F$-invariant  analytic discs, tangent to the axes, where $F$ acts as an irrational rotation. In particular, one can choose local coordinates at $(0,0)$, which we may assume to be defined on the unit ball $\B$ of $\C^2$ and $B\subset \B$, such that $\{z=0\}$ and $\{w=0\}$ are not contained in $V\cap \B$. Let $\B_\ast:=\B\setminus\{zw=0\}$.
Now, if $V\neq \Omega$, we can take $p_0\in\Omega$, $q_0\in V\setminus \Omega$, and $Z$ a connected open set containing $p_0$ and $q_0$ and such that $\overline Z\subset V$. Moreover, since $\{F^n\}$ converges uniformly to the origin on $\overline Z$, up to replacing $F$ by one of its iterates, we can assume that the forward $F$-invariant set $W :=\cup_{n\in\N}F^n(Z)$ satisfies $W\subset \B_\ast$. By construction, for every $\delta >0$ we can find $p\in Z\cap \Omega$ and $q\in Z\cap (V\setminus \Omega)$ such that $k_W(p,q)\le k_Z(p,q) <\delta$, where $k_W$ is the Kobayashi (pseudo)distance of $W$.
By the properties of the Kobayashi distance, for every $n\in\N$ we have
\[
k_{\B_\ast}(F^n(p), F^n(q))\leq k_W(p, q)<\delta.
\]
Also, if $(z_n,w_n):=F^n(p)$, $(x_n,y_n):=F^n(q)$, then
\[
k_{\D^\ast}(z_n,x_n)<\delta, \quad k_{\D^\ast}(w_n,y_n)<\delta,
\]
where $\D^\ast$ is the punctured unit disc. Since $q\not\in\Omega$, $F^n(q)\not \in B$ for all $n\in\N$, and so (by the above mentioned characterization of orbits' behavior of points in $\Omega$) we can ensure that, up to passing to a subsequence, we have $|x_n|\not\sim |y_n|$. By the triangle inequality and properties of the Kobayashi distance of $\D^\ast$, the shape of $B$ forces $k_{\D^\ast}(x_n,y_n)$ to be bounded from below by a constant depending only on $\beta$, leading to a contradiction (see Theorem \ref{Fatou-Omega} for details).

\smallskip 
Finally, in order to show that $\Omega$ is biholomorphic to $\C\times \C^\ast$ we construct a fibration from $\Omega$ to $\C$ in such a way that $\Omega$ is a line bundle minus the zero section over $\C$, hence, trivial. In fact, for this aim we do not need the Brjuno condition on $\lambda$.

We first prove in Section~\ref{local-coordi} the existence of a univalent map $Q$ on $B$ which intertwines $F$ on $B$ with a simple overshear. The first component $\psi$ of $Q$ is essentially the Fatou coordinate of the projection of $F$ onto the $zw$-plane and satisfies
\[
\psi \circ F=\psi+1.
\]
The second component $\sigma$ is the local uniform limit on $B$ of the sequence $\{\sigma_n\}$ defined by
\[
\sigma_n(z,w):= \lambda^n \pi_2(F^n(z,w)) \exp\left(\frac{1}{2}{\sum_{j=0}^{n-1} \frac{1}{\psi(z,w)+j}}\right),
\]
and satisfies the functional equation
\[
\sigma \circ F=\overline{\lambda}e^{-\frac{1}{2\psi}}\sigma.
\] 

Next, using dynamics, we extend such a map to a univalent map $G$ defined on a domain $\Omega_0\subset\Omega$, and we use it to prove that $\Omega$ is a line bundle minus the zero section over $\C$. Since all line bundles over $\C$ are globally holomorphically trivial, we obtain that $\Omega$ is biholomorphic to $\C\times\C^\ast$ (see   Section \ref{topology} for details).

We will now go through the proof in great detail in dimension $2$ and in the last section we will give the changes needed for the higher dimensional case.


\section*{Notations and conventions in $\C^2$}

We set up here some notations and conventions we shall use throughout the paper. 

We let $\pi:\C^2\to \C$,   $\pi_1:\C^2\to \C$,  $\pi_2:\C^2\to \C$ be defined by 
\[
\pi(z,w)=zw,\quad \pi_1(z,w)=z,\quad \pi_2(z,w)=w.
\]
If $F:\C^2\to \C^2$ is a holomorphic map, we denote by $F^n$ the $n$-th iterate of $F$, $n\in \N$, defined by induction as $F^n=F\circ F^{n-1}$, $F^0={\sf id}$. Moreover, for $(z,w)\in \C^2$ and $n\in \N$, we let
\[
u_n:=\pi(F^n(z,w)), \quad U_n:=\frac{1}{u_n}, \quad z_n:=\pi_1(F^n(z,w)), \quad w_n:=\pi_2(F^n(z,w)).
\]

If $f(n)$ and $g(n)$ are real positive functions of $n\in \N$, we write $f(n)\sim g(n)$, if there exist $0<c_1<c_2$ such that $c_1 f(n)<g(n)<c_2 f(n)$ for all $n\in \N$. Moreover, we use the Landau little/big ``O'' notations, namely, we write $f(n)=O(g(n))$, if there exists $C>0$ such that $f(n)\leq C g(n)$ for all $n\in \N$, while we write $f(n)=o(g(n))$, if $\lim_{n\to \infty}\frac{f(n)}{g(n)}=0$.

\section{The local basin of attraction $B$}\label{localB}

In this section we recall the construction of the local basin of attraction, and we provide the local characterization that we use in our construction.

\smallskip
Let $F_N$ be a germ of biholomorphism of $\C^2$, fixing the origin, of the form
\begin{equation}\label{Expression FN}
F_N(z,w)=\left(\lambda z\left(1 - \frac{zw}{2}\right), \overline{\lambda}w \left(1 - \frac{zw}{2}\right)\right),
\end{equation}
where $\lambda\in \C$, $|\lambda|=1$, is not a root of unity.

\begin{definition}
For $\theta\in (0,\frac{\pi}{2})$ and $R>0$ we let 
\[
S(R,\theta):=\left\{\zeta\in \C: \left|\zeta-\frac{1}{2R}\right|<\frac{1}{2R},  \ \  |{\sf Arg} (\zeta)|<\theta\right\}.
\]
Also,  we let
\[
H(R,\theta):=\{\zeta\in \C: \Re \zeta>R, |{\sf Arg}(\zeta)|<\theta\}.
\]

\end{definition}

D. Zaitsev and the first named author proved that any small variation of $F_N$ admits a local basin of attraction. In order to state the result in our case, let us introduce some sets:

\begin{definition}
For $\beta\in (0,\frac{1}{2})$ we let
\[
W(\beta):=\{(z,w)\in \C^2: |z|<|zw|^\beta,\ \  |w|<|zw|^\beta\}.
\]
For every $R\geq 0$, $\beta\in (0,\frac{1}{2})$ and $\theta\in (0,\frac{\pi}{2})$, we let 
\[
B(\beta, \theta, R):=\{(z,w)\in W(\beta): zw\in S(R, \theta)\}.
\]
\end{definition}

In \cite[Theorem 1.1]{BZ} it is proven:

\begin{theorem}\label{Thm:BZ}
Let $F_N$ be a germ of biholomorphism at $(0,0)$ of the form \eqref{Expression FN}. Let $\beta_0\in (0,1/2)$ and let $l\in \N$, $l\geq 4$ be such that $\beta_0(l+1)\geq 4$. Then for every $\theta_0\in(0, \pi/2)$ and for any germ of biholomorphism $F$ at $(0,0)$ of the form
\[
F(z,w)=F_N(z,w)+O(\|(z,w)\|^l)
\] 
there exists $R_0>0$ such that the (non-empty) open set $B_{R_0}:=B(\beta_0,\theta_0, R_0)$ is a uniform local basin of attraction for  $F$, that is $F(B_{R_0})\subseteq B_{R_0}$, and $\lim_{n\to \infty}F^n(z,w)=(0,0)$ uniformly in $(z,w)\in B_{R_0}$. 
\end{theorem}

\begin{definition} 
Let $F(z,w)=F_N(z,w)+O(\|(z,w)\|^l)$ be as in Theorem \ref{Thm:BZ} and fix $\theta_0\in(0, \pi/2)$. We set
\[
B:=B_{R_0}=B(\beta_0,\theta_0, R_0).
\]
\end{definition} 
 
In the following, we shall use some properties of $B$, that we prove below. We start with a lemma, allowing us to characterize the pre-images of $B$.

\begin{lemma}\label{go-good-down}
Let $F$ and $B$ be as in Theorem \ref{Thm:BZ}. Let $\beta\in (0,\frac{1}{2})$ be such that $\beta(l+1)> 2$ and $(z,w)\in \C^2$ such that $(z_n,w_n)\to (0,0)$ as $n\to \infty$. If there exists $n_0\in \N$ such that $(z_n,w_n)\in W(\beta)$ for all $n\geq n_0$, then 
\begin{enumerate}
\item $\lim_{n\to \infty}nu_n=1$  and $\lim_{n\to \infty}\frac{u_n}{|u_n|}=1$ (in particular, $|u_n|\sim \frac{1}{n}$),
\item $|z_n| \sim n^{-1/2}$ and $|w_n| \sim n^{-1/2}$,
\item for every $\gamma\in (0,1/2)$ there exists $n_\gamma\in \N$ such that $(z_n,w_n)\in W(\gamma)$ for all $n\geq n_\gamma$.
\end{enumerate}
In particular, $(z_n,w_n)\in B$ eventually.
\end{lemma}

\begin{proof} We can locally write $F$ in the form 
\begin{equation}\label{automFp-local}
F(z,w)=\left(\lambda z\left (1-\frac{zw}{2} \right)+R_l^1(z,w), \overline{\lambda} w\left (1-\frac{zw}{2} \right)+ R_l^2(z,w)\right),
\end{equation}
where $R_l^j(z,w)=O(\|(z,w)\|^l)$, $j=1,2$.

Since
$(z_n,w_n)\to (0,0)$, we have
\[
U_{n+1}=U_n\left(1+\frac{1}{U_n}+O\left(\frac{1}{|U_n|^2}, |U_n|\|(z_n,w_n)\|^{l+1}\right)\right).
\]
For $n\geq n_0$, $O(\|(z_n,w_n)\|^{l+1})$ is at most an $O(|u_n|^{\beta(l+1)})= O\left(\frac{1}{|U_n|^{\beta(l+1)}}\right)$, since $\beta(l+1)> 2$. Hence,
\begin{equation}\label{eq:U-n-n1}
U_{n+1}=U_n\left(1+\frac{1}{U_n}+O\left(\frac{1}{|U_n|^{\beta(l+1)-1}}, \frac{1}{|U_n|^2}\right)\right).
\end{equation}
Fix $\epsilon>0$. Let  $c:=1+\epsilon$. Notice that, by \eqref{eq:U-n-n1}, there exists $n_c\geq n_0$  such that for all $n\geq n_c$, $\left|U_{n+1}-U_n-1\right|<(c-1)/c$.
Arguing by induction on $n$, it easily follows that for all $n\geq n_c$ we have
\begin{equation}\label{eq:calim1}
\Re U_n\geq \Re U_{n_c}+\frac{n-n_c}{c},
\end{equation}
and 
\begin{equation}\label{eq:calim2}
|U_n|\leq |U_{n_c}|+c(n-n_c).
\end{equation}
Letting $\epsilon\to 0^+$ we obtain that
\begin{equation}\label{U-n-goes0}
\lim_{n\to \infty} \frac{\Re U_n}{n}=\lim_{n\to \infty} \frac{| U_n|}{n}=1.
\end{equation}
In particular, this means that $\lim_{n\to \infty} n\Re u_n=\lim_{n\to \infty}n |u_n|=1$. Hence, $\lim_{n\to \infty}\frac{|u_n|}{\Re u_n}=1$, which implies at once that 
\begin{equation}\label{Eq:go-right-up}
\lim_{n\to \infty} \frac{\Im u_n}{\Re u_n}=0.
\end{equation}
Hence statement (1) follows.

Arguing by induction, we have
\begin{equation}\label{forma-wn}
\begin{split}
z_{n+1}&=z_0\lambda^n \prod_{j=0}^n \Big(1-\frac{u_j}{2}\Big)+\sum_{j=0}^n R_l^1(z_j,w_j)\prod_{k=j+1}^n \lambda \Big(1-\frac{u_k}{2}\Big), \\
w_{n+1}&=w_0 \overline{\lambda}^n \prod_{j=0}^n \Big(1-\frac{u_j}{2}\Big)+\sum_{j=0}^n R_l^2(z_j,w_j)\prod_{k=j+1}^n \overline{\lambda} \Big(1-\frac{u_k}{2}\Big),
\end{split}
\end{equation}
Therefore,
\begin{equation}\label{how-z-n}
 |z_{n+1}|\leq |z_0|  \prod_{j=0}^n \Big|1-\frac{u_j}{2}\Big|+\sum_{j=0}^n |R_l^1(z_j,w_j)|\prod_{k=j+1}^n \Big|1-\frac{u_k}{2}\Big|.
\end{equation}

Taking into account statement (1), we have
\begin{equation*}
\begin{split}
\lim_{j\to \infty}(-2j)\log \Big|1-\frac{u_j}{2}\Big|
&=\lim_{j\to \infty}(-2j)\left(\frac{1}{2}\log \Big|1-\frac{u_j}{2}\Big|^2\right)\\
&=\lim_{j\to \infty}(-2j)\left( \frac{1}{8}|u_j|^2-\frac{1}{2}\Re u_j\right)=1.
\end{split}
\end{equation*}
Therefore,
\begin{equation}\label{Eq:prod1a}
\prod_{j=0}^n \Big|1-\frac{u_j}{2}\Big|=\exp\left(\sum_{j=0}^n \log \Big|1-\frac{u_j}{2}\Big|\right)\sim \exp\left(\sum_{j=1}^n -\frac{1}{2j}\right)\sim \frac{1}{\sqrt{n}}.
\end{equation}
Moreover, since $(z_n,w_n)\in W(\beta)$ eventually, and $|R_l^1(z_j,w_j)|=O(\|(z_j,w_j)|^l)$,  it follows that there exist some constants $0<c\leq  C$ such that 
\[
|R_l^1(z_j,w_j)|\leq c |u_j|^{\beta l}\leq C j^{-\beta l}.
\]
Hence, by \eqref{Eq:prod1a} we have for $j>1$ sufficiently large
\begin{equation*}
\begin{split}
|R_l^1(z_j,w_j)|\prod_{k=j+1}^n \Big|1-\frac{u_k}{2}\Big|
&=|R_l^1(z_j,w_j)|\exp\left(\sum_{k=j+1}^n\log \Big|1-\frac{u_k}{2}\Big| \right)\\
&\sim |R_l^1(z_j,w_j)|\exp\left(-\frac{1}{2}\sum_{k=j+1}^n\frac{1}{k} \right)\\ 
&\sim |R_l^1(z_j,w_j)| \frac{\sqrt{j}}{\sqrt{n}}\leq C\frac{j^{\frac{1}{2}-\beta l}}{\sqrt{n}}.
\end{split}
\end{equation*}
Since $\beta l-\frac{1}{2}>1$, it follows that  there exists a constant (still denoted by) $C>0$ such that
\[
\sum_{j=0}^n |R_l^1(z_j,w_j)|\prod_{k=j+1}^n \Big|1-\frac{u_k}{2}\Big|\leq C \frac{1}{\sqrt{n}}.
\]
Hence, from \eqref{how-z-n}, there exists a constant $C>0$ such that
\begin{equation}\label{Eq-z-n-goes}
|z_n|\leq C \frac{1}{\sqrt{n}}.
\end{equation}
A similar argument for $w_n$, shows that 
\begin{equation}\label{Eq-w-n-goes}
|w_n|\leq C \frac{1}{\sqrt{n}}.
\end{equation}
By statement (1), it holds  $|z_n|\cdot |w_n|=|u_n| \sim \frac{1}{n}$. Since $|z_n|\leq C\frac{1}{\sqrt{n}}$  and $|w_n|\leq C\frac{1}{\sqrt{n}}$ by \eqref{Eq-z-n-goes} and \eqref{Eq-w-n-goes}, it follows that, in fact, $|z_n|\sim \frac{1}{\sqrt{n}}$ and $|w_n|\sim \frac{1}{\sqrt{n}}$, proving statement (2).

Finally, by statement (2), there exist  constants $c, C>0$ such that $|z_n|\leq C\frac{1}{\sqrt{n}}$ for all $n\in \N$ and $|u_n|\geq c\frac{1}{n}$. Fix $\gamma\in (0,1/2)$. Then for every  $n$ large enough
\[
|z_n|\leq C \frac{1}{\sqrt{n}}\leq \frac{C}{c^{1/2}}|u_n|^{1/2}<|u_n|^\gamma.
\]
Similarly, one can prove that $|w_n|<|u_n|^{\gamma}$. As a consequence, eventually $(z_n,w_n)$ is contained in $W(\gamma)$ for every $\gamma\in (0,1/2)$.
\end{proof}

\begin{remark}\label{uniform-convergence-onB}
From the uniform convergence of $\{F^n\}$ to $(0,0)$ in $B$, and from the proof of the previous lemma, it follows that (1) and (2) in Lemma \ref{go-good-down} are uniform in $B$. \end{remark}

We shall also need the following local result concerning the topology of $B$:

\begin{lemma}\label{Omega}
Let $F$ and $B$ be as in Theorem \ref{Thm:BZ}. Then $ B$ is a doubly connected domain ({\sl i.e.}, $B$ is connected and its fundamental group is $\Z$). 
\end{lemma}
\begin{proof}
Let $\Phi:\C^2\to \C^2$ be defined by 
\begin{equation*}\label{Phi}
\Phi(z,w)=(zw,w).
\end{equation*}
The thesis then follows since $\Phi\colon B\to \Phi(B)$ is a biholomorphism and a straightforward computation shows that
\begin{equation*}\label{imagePhi}
\Phi( B)=\{(x,y)\in \C\times\C^*: x\in S(R, \theta), |x|^{1-{\beta_0}}<|y|<|x|^{\beta_0}\}.
\end{equation*}
\end{proof}

\section{Local Fatou coordinates on $B$}\label{local-coordi}

In this section we introduce special coordinates on $B$, which will be used later on in our construction. The first coordinate was introduced in \cite[Prop. 4.3]{BRZ}. Here we shall need more precise information, that is the following result:

\begin{proposition}\label{BRZ}
Let $F$ and $B$ be as in Theorem  \ref{Thm:BZ}. Then  there exists a holomorphic function $\psi:  B\to \C$ such that 
\begin{equation}\label{func-psi}
\psi \circ F = \psi +1.
\end{equation}
Moreover
\begin{equation}\label{psi-form}
\psi(z,w)=\frac{1}{zw}+c\log \frac{1}{zw}+v(z,w),
\end{equation}
where  $c\in \C$ depends only on $F_N$, and $v:B\to \C$ is a holomorphic  function such that for every $(z,w)\in B$,
\begin{equation}\label{v-gozero}
v(z,w)=zw\cdot g(z,w),
\end{equation}
for a bounded  holomorphic function $g:B\to \C$.
\end{proposition}

\begin{proof}
The strategy of the proof follows the one for the existence of Fatou coordinates in the Leau-Fatou flower theorem. Given a point $(z,w)\in B$, for all $n\in \N$ we have
\[
U_{n+1}=U_n+1+\frac{c}{U_n}+O(|U_n|^{-2})
\]
where $c\in\C$ depends on $F_N$ and, as usual, $U_n:=\frac{1}{\pi(F^n(z,w))}$. The map $\psi$ is then obtained as the uniform limit in $B$ of the sequence of functions $\{\psi_m\}_{m\in \N}$, where $\psi_m\colon B\to \C$ is defined as
\begin{equation}\label{approxFatou}
\psi_m(z,w):=\frac{1}{\pi(F^m(z,w))} - m-c\log \pi(F^m(z,w)).
\end{equation}
In fact, a direct computation as in \cite[Prop. 4.3]{BRZ} implies that there exists $A>0$ such that for all $m\in \N$ and all $(z,w)\in B$, 
\begin{equation}\label{psi-meno-psi}
|\psi_{n+1}(z,w)-\psi_{n}(z,w)|\leq A|U_n|^{-2}.
\end{equation}
Therefore, since $|U_n|=1/|u_n|\sim n$ uniformly in $B$ by Lemma \ref{go-good-down}  and Remark \ref{uniform-convergence-onB}, the sequence $\sum_{j=0}^n(\psi_{j+1}-\psi_j)$ is uniformly converging in  $B$ to a bounded holomorphic function $v$, that is,
\[
v(z,w):=\sum_{j=0}^\infty (\psi_{j+1}(z,w)-\psi_j(z,w)).
\] 
Moreover, \eqref{psi-form} follows from $\psi_n-\psi_0=\sum_{j=0}^n (\psi_{j+1}-\psi_j)$, and $\psi_n \circ F=\psi_{n+1}+1$ yields the functional equation \eqref{func-psi}. Notice that \eqref{psi-meno-psi} implies $|\psi-\psi_m|\sim \frac{1}{m}$.

Finally, since $U_n\in H(R_0,\theta_0)$ for all $n\in \N$, there exists $K\in (0,1)$ such that $\Re U_0>K|U_0|$ for all $U_0\in H(R_0,\theta_0)$. Hence, by \eqref{eq:calim1},
\begin{equation*}
\begin{split}
|v(z,w)|
&
\leq A\sum_{j=0}^\infty \frac{1}{|U_j|^2}
\leq A\sum_{j=0}^\infty \frac{1}{(\Re U_j)^2}
\leq A\sum_{j=0}^\infty\frac{1}{(\Re U_0+\frac{j}{2})^2}\\&
\sim A\int_0^\infty\frac{dt}{(\Re U_0+\frac{t}{2})^2} =\frac{2A}{\Re U_0}\leq \frac{2A}{K |U_0|},
\end{split}
\end{equation*}
from which \eqref{v-gozero} follows at once.
\end{proof}

\begin{definition}
The map $\psi: B \to \C$ is called a {\sl Fatou coordinate} for $F$. 
\end{definition}

\begin{lemma}\label{Lem-psi-quasi-inj}
Let $F$  be as in Theorem  \ref{Thm:BZ}. Let $\psi$ be the Fatou coordinate for $F$ given by Proposition \ref{BRZ}. Then there exist $R_1\geq R_0$, $\beta_1\in (\beta_0, \frac{1}{2})$  and $0<\theta_1<\theta_0$ such that  the holomorphic map 
\[
B(\beta_1, \theta_1, R_1)\ni (z,w)\mapsto (\psi(z,w), w)
\] 
is injective.
\end{lemma}

\begin{proof} 
First we search for $\beta_1$, $\theta_1$ and $R_1$ so that on $B(\beta_1, \theta_1, R_1)$ we have good estimates for the partial derivatives of $g$ and $v$ with respect to $U$.


Since the map $\chi\colon B\ni (z,w)\mapsto (U, w)$ is univalent, we can consider $v$ as a function of $(U,w)$ defined on 
\[
\chi(B)=\{(U,w): U\in H(R_0,\theta_0), |U|^{\beta_0-1}<|w|<|U|^{-\beta_0}\}.
\] 

Denote by $(H(R_0,\theta_0)+1)$ the set of points $U=V+1$ with $V\in H(R_0, \theta_0)$. Let $\theta_1\in (0,\theta_0)$ be such that $H(R_0+1,\theta_1)\subset (H(R_0, \theta_0)+1)$.
There exists $\delta_0>0$ such that for every $U\in (H(R_0,\theta_0)+1)$ the distance of $U$ from $\partial H(R_0, \theta_0)$ is greater than $2\delta_0$. 

Let $\tilde\beta\in (\beta_0,\frac{1}{2})$. For $R\geq R_0$, we have 
\[
\chi(B(\tilde\beta, \theta_1, R))
=
\{(U,w): U\in H(R,\theta_1), |U|^{\tilde\beta-1}<|w|<|U|^{-\tilde\beta}\},
\]
and
there exists $\tilde R\geq R_0$ such that for all $(U,w)\in \chi(B(\tilde\beta, \theta_1, \tilde R))$ and all $t\in\R$ it holds
\[
|U+\delta_0e^{it}|^{\beta_0-1}\le(|U|-\delta_0)^{\beta_0-1}<|U|^{\tilde\beta-1}<|w|<|U|^{-\tilde\beta}<(|U|+\delta_0)^{-\beta_0}\le|U+\delta_0e^{it}|^{-\beta_0},
\]
which implies that $(U+\delta_0 e^{it}, w)\in \chi(B)$ for all $t\in \R$ and all $(U,w)\in \chi(B(\tilde\beta, \theta_1, \tilde R))$, since $U+\delta_0e^{it}\in H(R_0, \theta_0)$ for all $t\in \R$. 
Therefore, for every $(U_0, w_0)\in \chi(B(\tilde\beta, \theta_1, \tilde R))$, the Cauchy formula for derivatives yields
\[
\left|\frac{\partial g}{\partial U}(U_0,w_0)  \right|=\frac{1}{2\pi}\left|\int_{|\zeta-U_0|=\delta_0}\frac{g(\zeta, w_0)}{(\zeta-U_0)^2}d\zeta\right|
\leq \frac{1}{2\pi\delta_0}\sup_{(U,w_0)\in \chi(B)}|g(U,w_0)|\leq \frac{C}{2\pi \delta_0}=:C_1.
\]
Hence, setting $C_2:=C+C_1$, for every $R\geq \min\{\tilde R, 1\}$, we have 
\begin{equation}\label{fuuuf}
\left| \frac{\partial v}{\partial U}(U_0,w_0)\right|
\leq \frac{C}{|U_0|^2}+\frac{1}{|U_0|}\left|\frac{\partial g}{\partial U}(U_0,w_0)\right|
\leq \frac{C}{|U_0|^2}+\frac{C_1}{|U_0|}
\leq \frac{C_2}{R}
\end{equation}
for all $(U_0, w_0)\in \chi(B(\tilde\beta, \theta_1, R))$.
Now, since there exists $K\in (0,1)$ such that $\Re U> K |U|$ 
for every $U\in H(\theta_1, R)$, we fix $\beta_1\in (\tilde\beta,\frac{1}{2})$ and let $R\geq \tilde R$ be such that 
\begin{equation}\label{t-subj}
K^{1-\beta_1}r^{\beta_1-1}>r^{\tilde\beta-1} \quad \forall r\geq R.
\end{equation}

To prove the injectivity on $B(\beta_1, \theta_1, R_1)$, we first prove that for each $(U_1, w_0), (U_2, w_0)\in \chi(B(\beta_1, \theta_1,  R))$ we have $(\gamma(t), w_0)\in \chi(B(\tilde\beta, \theta_1,  R))$ where $\gamma(t)= tU_1+(1-t)U_2$ with $t\in [0,1]$ is the real segment joining $U_1$ and $U_2$. In fact, we have $\gamma(t)\subset H(\theta_1, R)$ for all $t\in[0,1]$ since $H(\theta_1, R)$ is convex. Moreover, since $|U_j|>|w_0|^{\frac{1}{\beta_1-1}}$ and $\Re U_j>K|U_j|$ for $j=1,2$, we have
\begin{equation*}
|t U_1+(1-t)U_2|
>t\Re U_1+(1-t)\Re U_2
>K\left(t|w_0|^{\frac{1}{\beta_1-1}}+(1-t)|w_0|^{\frac{1}{\beta_1-1}}\right)
=K|w_0|^{\frac{1}{\beta_1-1}},
\end{equation*}
for all $t\in [0,1]$, and so, by \eqref{t-subj},
\[
|w_0|>\left(\frac{1}{K} \right)^{\beta_1-1}|t U_1+(1-t)U_2|^{\beta_1-1}>|t U_1+(1-t)U_2|^{\tilde\beta-1}.
\]
On the other hand, since $|U_j|<|w_0|^{-\frac{1}{\beta_1}}$, $j=1,2$, for all $t\in [0,1]$ we have,
\[
|tU_1+(1-t)U_2|
<t|w_0|^{-\frac{1}{\beta_1}}+(1-t)|w_0|^{-\frac{1}{\beta_1}}
=|w_0|^{-\frac{1}{\beta_1}},
\]
hence, 
\[
|tU_1+(1-t)U_2|^{-\tilde\beta}>|tU_1+(1-t)U_2|^{-\beta_1}>|w_0|.
\]
Therefore using \eqref{fuuuf} we obtain 
\begin{equation*}
\begin{split}
|\psi(U_1,w_0)-\psi(U_2,w_0)|
&=\left|\int_\gamma\frac{\partial \psi}{\partial U}(U,w_0)dU  \right|=\left|\int_\gamma \left[1+\frac{c}{U}+\frac{\partial v}{\partial U}(U,w_0)\right]dU  \right|\\
&\geq |U_1-U_0|-\frac{|c|}{R}|U_1-U_0|-\frac{C_2}{R}|U_1-U_0|\\
&=\left(1-\frac{|c|}{R}-\frac{C_2}{R}\right)|U_1-U_0|,
\end{split}
\end{equation*}
and we obtain the injectivity of $(U,w)\mapsto (\psi(U,w),w)$ on $\chi(B(\beta_1, \theta_1,  R))$, and hence of $(z,w)\mapsto (\psi(z,w), w)$ on $B(\beta_1,\theta_1, R)$, for $R$ sufficiently large.
\end{proof}

The next result shows the existence of another ``coordinate'' on $B$ defined using the Fatou coordinate. 

\begin{proposition}\label{Prop:second-local-coord}
Let $F$ and $B$ be as in Theorem \ref{Thm:BZ} and $\psi$ the Fatou coordinate given by Proposition \ref{BRZ}. Then there exists a holomorphic function $\sigma\colon  B\to \C^\ast$ such that 
\begin{equation}\label{func-sigma}
\sigma \circ F=\overline{\lambda}e^{-\frac{1}{2\psi}}\sigma.
\end{equation}
Moreover, $\sigma(z,w) = w + \eta(z,w)$,
where $\eta\colon B\to \C$ is a holomorphic  function such that for every $(z,w)\in B$ 
\begin{equation}\label{eta-gozero}
\eta(z,w)=(zw)^\alpha \cdot h(z,w),
\end{equation}
for a holomorphic bounded function  $h\colon B\to \C$, with $\alpha\in (1-\beta_0,1)\subset (1/2,1)$.
\end{proposition}

\begin{proof}
For $n\in \N$, consider the holomorphic function $\sigma_n \colon  B \to \C^\ast$ defined by
\begin{equation}\label{sigma-n}
\sigma_n(z,w):= \lambda^n \pi_2(F^n(z,w)) \exp\left({\frac{1}{2}\sum_{j=0}^{n-1} \frac{1}{\psi(z,w)+j}}\right).
\end{equation}
We will prove that the sequence $\{\sigma_n\}$ converges uniformly in $B$ to a holomorphic function $\sigma\colon  B\to \C^\ast$ satisfying the assertions of the statement.

First, if $\{\sigma_n\}$ is uniformly convergent on compacta of $ B$, then \eqref{func-sigma} follows from
\begin{equation*}
\begin{split}
\sigma_n \circ F
&=\lambda^n w_{n+1} \exp\left(\frac{1}{2}\sum_{j=0}^{n-1}\frac{1}{\psi\circ F+j}  \right)
=\lambda^n w_{n+1} \exp\left(\frac{1}{2}\sum_{j=0}^{n-1}\frac{1}{\psi+j+1}  \right)
\\
&=\overline{\lambda}\exp\left(-\frac{1}{2\psi} \right)\lambda^{n+1}w_{n+1}\exp\left(\frac{1}{2}\sum_{j=0}^{n}\frac{1}{\psi+j}  \right)=\overline{\lambda}\exp\left(-\frac{1}{2\psi} \right)\sigma_{n+1}.
\end{split}
\end{equation*}

Now we show that $\{\sigma_n\}$ is equibounded in $B$. By Proposition \ref{BRZ} we have 
$$
\left|\psi-\frac{1}{u_j}+j+c\log u_j\right|= |\psi-\psi_j|\sim \frac{1}{j}.
$$
By Lemma~\ref{go-good-down} and Remark \ref{uniform-convergence-onB}, $|u_j|\sim \frac{1}{j}$ uniformly in $B$, hence,   
\begin{equation}\label{eq:psi-goes-j}
\frac{1}{\psi+j}= \frac{u_j}{1-cu_j\log u_j +O(u_j)}=u_j+O(u_j^2\log u_j).
\end{equation}
Now, by statement (1) in Lemma \ref{go-good-down}, we have that 
$\lim_{j\to \infty}\frac{1}{2} j\Re u_j
=\frac{1}{2}$.
Therefore,
\[
\exp\left(\frac{1}{2}\sum_{j=0}^{n-1} \Re u_j\right)\sim \exp\left(\sum_{j=1}^{n-1} \frac{1}{2j}\right)=O(n^{1/2}).
\]
Moreover, again thanks to Lemma~\ref{go-good-down} and Remark \ref{uniform-convergence-onB}, there exists $C>0$ such that  $\sum_{j=0}^\infty |u_j^2\log u_j|\leq C$. Hence, there exists $C'>0$ such that
\begin{equation}\label{Eq:zero}
\begin{split}
\left|\exp\left(\frac{1}{2}{\sum_{j=0}^{n-1} \frac{1}{\psi(z,w)+j}}\right)\right|
&=\left| \exp\left({\sum_{j=0}^{n-1} \left( \frac{u_j}{2}+O(u_j^2\log u_j)\right)}\right)  \right|\\
&\leq  C'\exp\left(\sum_{j=1}^{n-1} \frac{1}{2j}\right)=O(n^{1/2}).
\end{split}
\end{equation}
Therefore, since $|w_n|\sim n^{-1/2}$, 
we have
\begin{equation}\label{Eq:1a}
|\sigma_n(z,w)|
=
|w_n| \left|\exp\left(\frac{1}{2}{\sum_{j=0}^{n-1} \frac{1}{\psi(z,w)+j}}\right)\right|=|w_n|O(n^{1/2})=O(1),
\end{equation}
showing that the sequence $\{\sigma_n\}$ is equibounded on $B$. 

To prove that $\{\sigma_n\}$ is in fact convergent, let us first notice that we have 
\begin{equation*}
\begin{aligned}
\sigma_{n+1}(z,w)
&= \lambda^{n+1} w_{n+1} \exp\left({\frac{1}{2}\sum_{j=0}^{n} \frac{1}{\psi(z,w)+j}}\right)\\
&= \lambda^{n+1} \left[\bar\lambda w_{n}\left(1-\frac{u_n}{2}\right) + R^2_l(z_n, w_n)\right] \exp\left({\frac{1}{2}\sum_{j=0}^{n} \frac{1}{\psi(z,w)+j}}\right)\\
&=\sigma_n(z, w) \left(1-\frac{u_n}{2}\right)  e^{\frac{1}{2(\psi(z,w)+n)}} + \lambda^{n+1} R^2_l(z_n, w_n)\exp\left(\frac{1}{2}{\sum_{j=0}^{n} \frac{1}{\psi(z,w)+j}}\right).
\end{aligned} 
\end{equation*}
Therefore,
\begin{equation}\label{Eq:diff-sigman}
\begin{split}
\sigma_{n+1}(z,w) - \sigma_{n}(z,w)
&=
\sigma_n(z, w) \left[\left(1-\frac{u_n}{2}\right)  e^{\frac{1}{2(\psi(z,w)+n)}} -1\right] \\
&\quad+ \lambda^{n+1} R^2_l(z_n, w_n)\exp\left({\frac{1}{2}\sum_{j=0}^{n} \frac{1}{\psi(z,w)+j}}\right).
\end{split}
\end{equation}

Now we estimate the terms in the right hand side of \eqref{Eq:diff-sigman}. Fix $\alpha \in (1-\beta_0,1)$. Note that $\alpha>\frac{1}{2}$. By \eqref{eq:psi-goes-j}, recalling that $|u_n|\sim \frac{1}{n}$, we have
\begin{equation}\label{Eq:1b}
\begin{split}
\left(1-\frac{u_j}{2}\right)  e^{\frac{1}{2(\psi(z,w)+n)}} -1
&=\left(1-\frac{u_j}{2}\right) e^{\frac{1}{2}u_n+O(u_n^2\log u_n)}-1\\
&=\left(1-\frac{u_j}{2}\right)\left(1+\frac{1}{2} u_n +O(u_n^2\log u_n)\right)-1\\
&=
O(u_n^2\log u_n)
=|u_n|^\alpha O\left(\frac{\log n}{n^{2-\alpha}}\right).
\end{split}
\end{equation}
Next, since $(z_n,w_n)\in B$, we have that $|R^2_l(z_n,w_n)|=O(|u_n|^{\beta_0 l})$, and by \eqref{Eq:zero}, we have
\begin{equation}\label{Eq:1c}
|R_l^2(z_n,w_n)|  \left|\exp\left({\frac{1}{2}\sum_{j=0}^{n} \frac{1}{\psi(z,w)+j}}\right)\right|\leq C |u_n|^\alpha n^{\frac{1}{2}+\alpha-\beta_0 l},
\end{equation}
for some $C>0$.

From \eqref{Eq:diff-sigman}, using \eqref{Eq:1a}, \eqref{Eq:1b}, \eqref{Eq:1c}, it follows that there exists a constant $C'>0$ such that for all $(z,w)\in B$,
\begin{equation}\label{Eq:sigma-good}
|\sigma_{n+1}(z,w) - \sigma_{n}(z,w)|
\le C_n|u_n|^\alpha,
\end{equation}
with $C_n=C' \left(\frac{\log n}{n^{2-\alpha}}+n^{\frac{1}{2}+\alpha-\beta_0 l}\right)$. Therefore the sequence $\{\sigma_n\}$ converges uniformly on $B$ to a holomorphic function $\sigma$.  
Let $C:=\sum_{n=0}^\infty C_n<+\infty$.  
For all $n\in\N$, we have $|u_n|\leq 1/R_0$, 
hence \eqref{Eq:sigma-good} implies that  $\sigma_n-\sigma_0=\sum_{j=0}^n (\sigma_{j+1}-\sigma_j)$ converges uniformly on  $B$ to a holomorphic function $\eta$ such that $\eta(z,w)= \sigma(z,w)-\sigma_0(z,w)=\sigma(z,w)-w$. 

Moreover, for all $(z,w)\in B$ we have
\[
|\eta(z,w)|\leq \sum_{n=0}^\infty |\sigma_{n+1}(z,w)-\sigma_n(z,w)|\leq \sum_{n=0}^\infty C_n |u_n|^\alpha<|u_0|^\alpha\sum_{n=0}^\infty C_n=C |zw|^\alpha.
\]
Finally, since $\sigma_n(z,w)\neq 0$ for all $n\in \N$ and $(z,w)\in B$, it follows that either $\sigma\equiv 0$ or $\sigma(z,w)\neq 0$ for all $(z,w)\in B$. Since $(r,r)\in B$ for all $r>0$ sufficiently small, recalling that we just proved that $\sigma(z,w)=w+(zw)^{\alpha} h(z,w)$, with $|h|\leq C$ for all $(z,w)\in B$, and $2\alpha>1$, we have
\[
|\sigma(r,r)|=|r+r^{2\alpha}h(r,r)|\geq r-r^{2\alpha}C=r(1-o(r)),
\]
proving that $\sigma\not\equiv 0$.
\end{proof}

We shall now prove that the map $B\ni (z,w)\mapsto (\psi(z,w), \sigma(z,w))$ is injective on a suitable subset of $B$. Such a result is crucial to show that the global basin of attraction which we shall introduce in the next section, is biholomorphic to $\C\times \C^\ast$.

\begin{proposition}\label{local-inj}
Let $F$ and $B$ be as in Theorem \ref{Thm:BZ}, let $\psi:B\to \C$ be the Fatou coordinate given by Proposition \ref{BRZ} and let $\sigma: B\to \C$ be the second local coordinate defined in Proposition \ref{Prop:second-local-coord}. Then there exist $R_1\geq R_0$, $\beta_1\in (\beta_0, \frac{1}{2})$  and $\theta_1\in(0,\theta_0]$  such that  the holomorphic map 
\[
B(\beta_1, \theta_1, R_1)\ni (z,w)\mapsto Q(z,w):=(\psi(z,w), \sigma(z,w))
\] 
is injective. 

Moreover, there exist  $\tilde R>1$, $\tilde \theta\in (0,\frac{\pi}{2})$ and $\tilde \beta\in (0,\frac{1}{2})$ such that 
\begin{equation}\label{inside-claim}
\left\{(U,w)\in\C^2: U\in H(\tilde R, \tilde\theta), |U|^{\tilde\beta-1}<|w|<|U|^{-\tilde\beta}\right\}
\subset Q(B).
\end{equation}
\end{proposition}

\begin{proof} 
Let $R_1\geq R_0$, $\beta_1\in (\beta_0, \frac{1}{2})$  and $0<\theta_1\leq\theta_0$ be given by Lemma \ref{Lem-psi-quasi-inj}. Thanks to the injectivity of $B(\beta_1, \theta_1, R_1)\ni (z,w)\mapsto (\psi(z,w), w)$ showed in Lemma \ref{Lem-psi-quasi-inj}, it follows easily that the map
\[
B(\beta_1, \theta_1, R_1)\ni (z,w)\mapsto (\psi(z,w), \sigma_n(z,w))
\]
is injective for all $n\in \N$, where $\sigma_n$ is the map defined in \eqref{sigma-n} for $n\in \N$. Since $\sigma$ is the uniform limit of the sequence $\{\sigma_n\}$, it follows that either the Jacobian of $Q=(\psi,\sigma)$ is identically zero on $B(\beta_1, \theta_1, R_1)$, or $Q$ is injective on $B(\beta_1, \theta_1, R_1)$. 

We now compute the Jacobian of $Q$ at $(r,r)\in B(\beta_1, \theta_1, R_1)$, for $r>0$, $r$ sufficiently small. To simplify computation, we consider the holomorphic change of coordinates $\chi\colon B(\beta_1, \theta_1, R_1)\to \C^2$ given by $\chi(z,w)=(\frac{1}{zw},w)=(U,w)$ and we compute the Jacobian of $Q(U,w)$ at $(\frac{1}{r^2}, r)$. 

By Proposition \ref{BRZ} and Proposition \ref{Prop:second-local-coord}, we have
\begin{equation}\label{Q-expr}
Q(U,w)=(U+c\log U+v(U,w), w+\eta(U,w)),
\end{equation}
where $v(U,w)=\frac{1}{U}g(U,w)$ and $\eta(U,w)=\frac{1}{U^\alpha}h(U,w)$, $\alpha\in (1-\beta_0,1)$, with $|g|, |h|\leq C$ for some $C>0$ on $B$. Hence,
\begin{equation*}
\begin{split}
{\sf Jac}_{\left(\frac{1}{r^2},r\right)}Q&=\det 
\left(\begin{matrix}
1+cr^2+\frac{\partial v}{\partial U}\left(\frac{1}{r^2},r\right)& \frac{\partial v}{\partial w}\left(\frac{1}{r^2},r\right)\\
\frac{\partial \eta}{\partial U}\left(\frac{1}{r^2},r\right)& 1+\frac{\partial \eta}{\partial w}\left(\frac{1}{r^2},r\right)
\end{matrix}\right)\\
&=\left(1+cr^2+\frac{\partial v}{\partial U}\left(\frac{1}{r^2},r\right)\right)\left(1+\frac{\partial \eta}{\partial w}\left(\frac{1}{r^2},r\right) \right)-\frac{\partial v}{\partial w}\left(\frac{1}{r^2},r\right)\frac{\partial \eta}{\partial U}\left(\frac{1}{r^2},r\right).
\end{split}
\end{equation*}
First of all, note that for  $\gamma\in (0,\frac{1}{2})$, $\tilde R>1$ and  $\tilde\theta\in (0,\frac{\pi}{2})$ there exists $r_0>0$ such that $(\frac{1}{r^2}, r)\in \chi((B(\gamma, \tilde\theta, \tilde R)))$ for all $r\in (0,r_0)$.  Hence, by \eqref{fuuuf}, there exists $C_2>0$ such that for  $r$ sufficiently small, 
\[
\left|\frac{\partial v}{\partial U}\left(\frac{1}{r^2},r\right)\right|\leq r^2C_2.
\]
A similar argument as in \eqref{fuuuf} for $\eta$ instead of $v$, shows that for  $r$ sufficiently small, 
\[
\left|\frac{\partial \eta}{\partial U}\left(\frac{1}{r^2},r\right)\right|\leq r^{2\alpha} C_3,
\]
for some $C_3>0$.

On the other end, it is easy to check that, for every $t\in \R$, $(\frac{1}{r^2}, r(1+\frac{e^{it}}{2}))\in \chi(B)$ whenever $r$ is positive and small enough. Hence, by the Cauchy formula for derivatives
\begin{equation*}
\left|\frac{\partial v}{\partial w}\left(\frac{1}{r^2},r\right)\right|
=\frac{1}{2\pi}\left|\int_{|\zeta-r|=r/2}\frac{v(\frac{1}{r^2},\zeta)}{(\zeta-r)^2}d\zeta\right|
=\frac{r^2 \max_{|\zeta-r|=r/2}|g(\frac{1}{r^2},\zeta)|}{r}\leq Cr.
\end{equation*}
Similarly,
\[
\left|\frac{\partial \eta}{\partial w}\left(\frac{1}{r^2},r\right)\right|\leq C r^{2\alpha-1}.
\]
Therefore,
\[
{\sf Jac}_{\left(\frac{1}{r^2},r\right)}Q=1+O(r^{2\alpha-1}),
\]
showing that the Jacobian is not zero for $r$ sufficiently small since $\alpha>1/2$. Hence $Q$ is injective on $B(\beta_1, \theta_1, R_1)$.

Now we prove  there exist $\tilde R>1$, $\tilde \theta\in (0,\frac{\pi}{2})$ and $\tilde \beta\in (0,\frac{1}{2})$ such that \eqref{inside-claim} holds. The rough idea is that $Q|_{B}$ is ``very close'' to the map $(z,w)\mapsto (\frac{1}{zw}-c\log (zw), w)$, for which the statement is true, and hence \eqref{inside-claim} follows by Rouch\'e's Theorem. 

Consider again the constants $R_1\geq R_0$, $\beta_1\in (\beta_0, \frac{1}{2})$  and $\theta_1\in(0,\theta_0]$ given by Lemma~\ref{Lem-psi-quasi-inj}, and the holomorphic change of coordinates on $\tilde B$ given by $\chi(z,w)=(\frac{1}{zw}, w)=(U,w)$. Then $\chi(\tilde B)=\{(U,w): U\in H(R_1,\theta_1), |U|^{\beta_1-1}<|w|<|U|^{-\beta_1}\}$. 

The map $\chi(\tilde B)\ni (U,w)\mapsto Q(U,w)=(\psi(U,w), \sigma(U,w))$ is given by \eqref{Q-expr}. In particular
\begin{equation}\label{good-exp}
\psi(U,w)=U(1+\tau(U,w)), 
\end{equation}
where $|\tau|<C$   on $\chi(\tilde B)$ for some $C>0$, and $\lim_{|U|\to \infty}\tau(U,w)=0$.
This implies immediately that there exist $\tilde R_1>0$ and $\tilde\theta\in (0,\frac{\theta_0}{2})$ such that $H(\tilde R_1,2\tilde\theta)\subset \psi(\tilde B)\subset \psi(B)$. 

To prove \eqref{inside-claim} it suffices to show that there exist $\tilde R\geq \tilde R_1$ and $\tilde\beta\in (\beta_1, \frac{1}{2})$ such that for every $\zeta_0\in H(\tilde R,\tilde\theta)$,  
\begin{equation}\label{last-claim}
\{\xi\in \C: |\zeta_0|^{\tilde\beta-1}<|\xi|<|\zeta_0|^{-\tilde\beta}\}\subset \sigma(\psi^{-1}(\zeta_0)).
\end{equation}

In order to prove \eqref{last-claim}, we first show that there exist $\tilde R_2\geq \tilde R_1$ and $\tilde \beta_2\in (\beta_1, \frac{1}{2})$ such that for every $\zeta_0\in H(\tilde R_2,\tilde\theta)$ it holds
\begin{equation}\label{fiber-pi2}
\{\xi\in \C: |\zeta_0|^{\tilde\beta_2-1}<|\xi|<|\zeta_0|^{-\tilde\beta_2}\}\subset \pi_2(\psi^{-1}(\zeta_0)).
\end{equation}
Indeed, by \eqref{good-exp},  $\zeta_0=\psi(U,w)=U(1+\tau(U,w))$ with $|\tau|<C$ and $\lim_{|U|\to \infty}\tau(U,w)=0$. Hence, if $\zeta_0\in H(\tilde R_2,\tilde\theta)$ for some $\tilde R_2\geq \tilde R_1$,
\[
|U|\geq \frac{|\zeta_0|}{1+|\tau(U,w)|}\geq \frac{\tilde R_2}{1+C}.
\]
Therefore, given $c'\in (0,1)$, we can choose $\tilde R_2\geq \tilde R_1$ large enough so that for every $(U,w)\in \chi(\tilde B)$ such that $\psi(U,w)=\zeta_0$ and $\zeta_0\in H(\tilde R_2,\tilde\theta)$, the modulus $|U|$ is so large that
$|\tau(U,w)|<c'$. This implies that
\begin{equation}\label{c-stima-U}
(1-c')|U|<|\zeta_0|<(1+c')|U|
\end{equation}
for every $U\in \C$ such that there exists $w\in \C$ so that $(U,w)\in \chi(\tilde B)$ and $\psi(U,w)=\zeta_0\in H(\tilde R_2,\tilde\theta)$.  

Let $\tilde \beta_2\in (\beta_1, \frac{1}{2})$. Let $r_0>0$ be such that
\[
\frac{1}{[(1+c')t]^{1-\beta_1}}<\frac{1}{t^{1-\tilde\beta_2}}<\frac{1}{[(1-c')t]^{\tilde\beta_2}}<\frac{1}{t^{\beta_1}}, \quad \forall t\geq r_0.
\] 
Up to choosing $\tilde R_2\geq r_0$, \eqref{c-stima-U} implies that 
\begin{equation}\label{one-another}
|U_0|^{\beta_1-1}<|\zeta_0|^{\tilde\beta_2-1}<|\zeta_0|^{-\tilde\beta_2}<|U_0|^{-\beta_1}
\end{equation}
for every $U_0\in \C$ such that there exists $w\in \C$ so that $(U_0,w)\in \chi(	\tilde B)$ and $\psi(U_0,w)=\zeta_0\in H(\tilde R_2,\tilde\theta)$. 

Fix $\zeta_0\in H(\tilde R_2,\tilde\theta)$ and fix $\xi_0\in \C$ such that  $|\zeta_0|^{\tilde\beta_2-1}<|\xi_0|<|\zeta_0|^{-\tilde\beta_2}$. Since there exists $(U_0,w_0)\in \chi(\tilde B)$ such that $\psi(U_0,w_0)=\zeta_0$, it follows from \eqref{one-another} that $(U_0,\xi_0)\in \chi(\tilde B)$.  In particular, $\chi(\tilde B)\cap \{w=\xi_0\}\neq \emptyset$. Set
\[
A(\xi_0):=
\left\{U\in H(R_1,\theta_1): \frac{1}{|\xi_0|^{\frac{1}{1-\beta_1}}}<|U|<\frac{1}{|\xi_0|^{\frac{1}{\beta_1}}}\right\}
=\chi(\tilde B)\cap \{w=\xi_0\}.
\]
Then, 
\[
A(\xi_0)\ni U\mapsto \psi_{\xi_0}(U):=\psi(U,\xi_0)=U+c\log U+\frac{g(U,\xi_0)}{U}\in \C
\]
is well defined and  holomorphic. Moreover, up to taking $\tilde R_2$ larger and $\tilde \theta$ smaller, we can assume that the set $H(\tilde R_2, \tilde\theta)$ is contained in the image of the map $\chi(\tilde B) \ni (U,w)\mapsto U+c\log U$. Hence, there exists $(U_0, w_0)\in \chi(\tilde B)$ such that  $U_0+c\log U_0=\zeta_0$. Since $\zeta_0=U_0(1+c\frac{\log U_0}{U_0})$, it follows that $|U_0|(1-\epsilon)\leq|\zeta_0|\leq |U_0|(1+\epsilon)$ for some $\epsilon\in (0,1)$, provided that $\tilde R_2$ is sufficiently large. Recalling that  $|\zeta_0|^{\tilde\beta_2-1}<|\xi_0|<|\zeta_0|^{-\tilde\beta_2}$, we have
\[
|U_0|\geq \frac{|\zeta_0|}{1+\epsilon} >\frac{1}{(1+\epsilon)|\xi_0|^{1/(1-\tilde\beta_2)}}>\frac{1}{|\xi_0|^{\frac{1}{1-\beta_1}}},
\]
where the last inequality holds provided $\tilde R_2$ is sufficiently large. Similarly, one can show that $|U_0|<\frac{1}{|\xi_0|^{\frac{1}{\beta_1}}}$, namely, $U_0 \in A(\xi_0)$. 
 
 Let $\delta\in (0,1)$ be such that $D(U_0,\delta):=\{U\in \C: |U-U_0|<\delta\}\subset A(\xi_0)$. Since $|g(U,\xi_0)|/|U|<c'$, up to choosing $\tilde R_2$ so large that $\displaystyle c'+\delta<|c|\max_{|U-U_0|=\delta}\left|\log{U}-\log{U_0}\right|$, it follows that for all $U\in \partial D(U_0,\delta)$,
\begin{equation*}
\begin{split}
|\psi_{\xi_0}(U)-U-c\log U|&<c'<|c|\left|\log\frac{U}{U_0}\right|-\delta\le|U+c\log U-\zeta_0|\\&\le|U+c\log U-\zeta_0|+|\psi_{\xi_0}(U)-\zeta_0|.
\end{split}
\end{equation*}
Hence, Rouch\'e's Theorem implies that there exists $U_1\in D(U_0,\delta)\subset A(\xi_0)$ such that $\psi(U_1,\xi_0)=\psi_{\xi_0}(U_1)=\zeta_0$, proving \eqref{fiber-pi2}. 

Let $K\colon\chi(\tilde B)\to \C^2$ be defined by $K(U,w):=(\psi(U,w), w)$. Then the map $K$ is injective and from \eqref{fiber-pi2}, we obtain that
\begin{equation}\label{XQR}
\chi(B(\tilde\beta_2,  \tilde\theta,  \tilde R_2))\subset K(\chi(\tilde B)).
\end{equation}
Let $\tilde R\geq \tilde R_2$, and let $\zeta_0\in H(\tilde R, \tilde \theta)$. Thanks to \eqref{XQR}, we have $(\zeta_0,w)\in K(\chi(\tilde B))$ for every $w\in J(\zeta_0)$, where
\[
J(\zeta_0):=\{w \in \C:  |\zeta_0|^{\tilde\beta_2-1}<|w|<|\zeta_0|^{-\tilde\beta_2}\}.
\]
 
Let $\tilde\beta\in (\tilde\beta_2,\frac{1}{2})$, and let $\xi_0\in \C$ be such that $ |\zeta_0|^{\tilde\beta-1}<|\xi_0|<|\zeta_0|^{-\tilde\beta}$. In particular $\xi_0\in J(\zeta_0)$, and setting $r:=\min\{|\zeta_0|^{\tilde \beta-1}-|\zeta_0|^{\tilde\beta_2-1}, |\zeta_0|^{-\tilde\beta_2}-|\zeta_0|^{-\tilde\beta}\}>0$, the disc $D(\xi_0,r):=\{\xi\in \C: |\xi-\xi_0|<r\}$ is contained in $J(\zeta_0)$. 
Moreover, if $\tilde R$ is sufficiently large,
\begin{equation}\label{r-minmin}
r>\frac{1}{2}\min\{|\zeta_0|^{\tilde \beta-1}, |\zeta_0|^{-\tilde \beta_2}\}.
\end{equation}

Set $(\tilde U,w):=K(U,w)$. For every $(\tilde U, w)\in K(\chi(\tilde B))$, we can write
\[
\tilde\sigma(\tilde U, w):=(\sigma \circ K^{-1})(\tilde U, w)=w+\eta(\tilde U,w),
\]
where $\eta(\tilde U,w)=\frac{1}{\tilde U^\alpha}h(\tilde U,w)$, with $\alpha\in (1-\beta_0,1)$, and $|h|\leq C$ for some $C>0$. By \eqref{r-minmin}, since $\alpha>1-\beta_0>1/2$, if $\tilde R$ is sufficiently large, then
$|\eta(\zeta_0, w)|<r$, for every $w\in J(\zeta_0)$.  Therefore,  for all $w\in \partial D(\xi_0,r)$, 
\[
|w-\tilde\sigma(\zeta_0, w)|=|\eta(\zeta_0,w)|< r=|w-\xi_0|\leq |w-\xi_0|+|\tilde\sigma(\zeta_0, w)-\xi_0|.
\]
Hence, by Rouch\'e's Theorem, there exists $w_0\in D(\xi_0,r)$ such that $\tilde\sigma(\zeta_0, w_0)=\xi_0$. By the arbitrariness of $\xi_0$, this implies that for every $\zeta_0\in H(\tilde R, \tilde\theta)$ 
\[
\{\xi\in \C: |\zeta_0|^{\tilde\beta-1}<|\xi|<|\zeta_0|^{-\tilde\beta}\}\subset \tilde\sigma(\zeta_0, \cdot)(J(\zeta_0))\subset \sigma(\psi^{-1}(\zeta_0)),
\] 
which finally proves \eqref{last-claim}.
\end{proof}
 
\section{The topology of the global basin $\Omega$}\label{topology}
 
Let $F_N$ be a germ of biholomorphism of $\C^2$ at $(0,0)$ of the form \eqref{Expression FN}. Thanks to a result of  B. J.  Weickert \cite{W1} and F. Forstneri\v{c} \cite{F} (see in particular \cite[Corollary 2.2]{F}), given any $l\geq 2$ there exists an automorphism $F$ of $\C^2$ such that $\|F(z,w)-F_N(z,w)\|=O(\|(z,w)\|^l)$. In particular, given $\lambda$ a unimodular number not a root of unit, we take $l\geq 4$  such that $\beta_0 (l+1)\geq 4$, where $0<\beta_0<1/2$ is given by Theorem \ref{Thm:BZ}, and we consider automorphisms of $\C^2$ of the form
\begin{equation}\label{automFp}
F(z,w)=\left(\lambda z\left(1-\frac{zw}{2} \right)+R_l^1(z,w), \overline{\lambda}w \left(1-\frac{zw}{2} \right)+ R_l^2(z,w)\right),
\end{equation}
where $R_l^j(z,w)=O(\|(z,w)\|^l)$, $j=1,2$.

\begin{definition}
Let $F$ be an automorphism of $\C^2$ of the form \eqref{automFp}. Let $B$ be the local basin of attraction of $F$ given by Theorem \ref{Thm:BZ}. The \emph{global attracting basin of $F$} is
\[
\Omega:=\bigcup_{n\in \N} F^{-n}(B). 
\]
\end{definition}

In this section we are going to prove that the global basin $\Omega$ is biholomorphic to $\C\times \C^\ast$. We start by proving that $\Omega$ is not simply connected: 

\begin{proposition}\label{prop2.1}
The open set $\Omega$ is connected but not simply connected.
\end{proposition}
\begin{proof} 
We see that $\Omega$ is the growing union of images biholomorphic to $B$ which is doubly connected by Lemma \ref{Omega}. Moreover, $F_*$ is the identity on $\pi_1(B)$ and on $H_1(B)$, therefore $\pi_1(\Omega) = H_1(\Omega) = \Z$.
\end{proof}

In order to prove that $\Omega$ is biholomorphic to $\C\times\C^\ast$, let us consider the Fatou coordinate $\psi$ for $F$ given by Proposition \ref{BRZ} and the holomorphic function $\sigma$ given by Proposition \ref{Prop:second-local-coord}. We can use the functional equation \eqref{func-psi} to extend $\psi$ to all $\Omega$. Indeed, let $p\in \Omega$. Then there exists $n\in \N$ such that $F^n(p)\in  B$. We define 
\[
g_1(p):=\psi(F^n(p))-n.
\]
Set $H:=g_1(B)$, and consider $\Omega_0:=g_1^{-1}(H)=\bigcup_{\zeta\in H}g_1^{-1}(\zeta)$. 

Using \eqref{func-sigma} we can extend $\sigma$ to $\Omega_0$ as follows.  For any $p\in \Omega_0$, we set
\begin{equation*}
\begin{split}
g_2(p)&:=\lambda^n \exp\left(\frac{1}{2}\sum_{j=0}^{n-1}\frac{1}{g_1(p)+j} \right)\sigma(F^n(p))\\
&=\lambda^n \exp\left(\frac{1}{2}\sum_{j=0}^{n-1}\frac{1}{\psi(F^n(p))+j-n} \right)\sigma(F^n(p)),
\end{split}
\end{equation*}
where $n\in \N$ is such that $F^n(p)\in  B$. Notice that, since $g_1(p)\in H$, we have $\Re g_1(p)>0$ and the previous formula is well defined.

The next lemma shows that the map $G:=(g_1, g_2)\colon \Omega_0\to \C^2$ is well defined and holomorphic: 

\begin{lemma}\label{lemma_g_1}
The map $G:=(g_1,g_2)\colon \Omega_0 \to \C^2$ is well-defined, holomorphic and injective. 
\end{lemma}

\begin{proof}
The map $G$ is holomorphic by construction and since $\Re g_1(p)>0$ for all $p\in\Omega_0$.

The map $G$ is well defined. Indeed, if $n$ and $m$ are both integers so that $F^n(p)$ and $F^m(p)$ belong to $ B$, and $n<m$, then $F^m(p) = F^{m-n}(F^n(p))$. Therefore $\psi(F^m(p)) = \psi(F^{m-n}(F^n(p))) = \psi(F^n(p)) + m-n$, whence $\psi(F^m(p))-m = \psi(F^n(p))
-n$. Analogously, $\sigma(F^m(p))= \overline{\lambda}^{m-n}\exp((1/2)\sum_{j=0}^{m-n-1}1/(\psi(F^n(p))+j))\sigma(F^n(p))$, and so
\begin{equation*}
\begin{aligned}
&\lambda^m \exp\left(\frac{1}{2}\sum_{j=0}^{m-1}\frac{1}{\psi(F^m(p))+j-m} \right)\sigma(F^m(p))
\\
&=
\lambda^m \exp\left(\frac{1}{2}\sum_{j=0}^{m-1}\frac{1}{\psi(F^n(p))+j-n} \right)\overline{\lambda}^{m-n}\exp\left(-\frac{1}{2}\sum_{j=0}^{m-n-1}\frac{1}{\psi(F^n(p))+j}\right)\sigma(F^n(p))\\
&=\lambda^n \exp\left(\frac{1}{2}\sum_{j=0}^{n-1}\frac{1}{\psi(F^n(p))+j-n} \right)\sigma(F^n(p)),
\end{aligned}
\end{equation*}
and we are done.

Let us now prove the injectivity of $G$. Let $p, q\in \Omega_0$. By the very definition of $G$, $G(p)=G(q)$ if and only if 
\[
(\psi(F^n(p)), \sigma(F^n(p)))=(\psi(F^n(q)), \sigma(F^n(q)))
\]
for all $n\in \N$ such that $F^n(p)$ and $F^n(q)$ are contained in $B$. By Proposition \ref{local-inj}, there exist $R_1\geq R_0$, $\beta_1\in (\beta_0, \frac{1}{2})$  and $0<\theta_1\leq\theta_0$  such that  $Q:=(\psi, \sigma)$ is injective on  $B(\beta_1, \theta_1, R_1)$. Also, by Lemma \ref{go-good-down}, there exists $n\in \N$ such that $F^n(p), F^n(q)\in B(\beta_1, \theta_1, R_1)$. Therefore, $G(p)= G(q)$ if and only if $p=q$.
 \end{proof}
 
\begin{proposition}\label{HxCstar}
$G(\Omega_0)=H\times \C^\ast$.
\end{proposition}
\begin{proof}
Let $T:\C^2\to\C^2$ be defined by 
\[
T(\zeta, \xi):=(\zeta+1,\overline{\lambda}e^{-\frac{1}{2\zeta}}\xi).
\]
Notice that $T$ is not defined at $\zeta=0$. However, since $g_1(\Omega_0)=H$, the map $T$ is well-defined and holomorphic on $G(\Omega_0)$ and satisfies
\[
G\circ F=T\circ G.
\]

Let $(\zeta_0, \xi_0)\in H\times \C^\ast$. By induction, for $n\in \N$, we have
\[
(\zeta_n,\xi_n)
:=
T^n(\zeta_0,\xi_0)=\left(\zeta_0+n,\overline{\lambda}^n\exp\left(-\frac{1}{2}\sum_{j=0}^{n-1}\frac{1}{\zeta_0+j} \right) \xi_0\right).
\] 
Now, 
\begin{equation*}
\begin{split}
|\xi_n|
&=\exp\left(-\frac{1}{2}\sum_{j=0}^{n-1}\Re\left(\frac{1}{\zeta_0+j} \right) \right)|\xi_0|\\
&=\exp\left(-\frac{1}{2}\sum_{j=1}^{n-1}\frac{1}{j}\left(\frac{1+j^{-1}\Re \zeta_0}{\left|j^{-1}\zeta_0+1\right|^2}\right)\right)\exp\left(-\Re\frac{\zeta_0}{2|\zeta_0|^2} \right)|\xi_0|,
\end{split}
\end{equation*}
which implies that  
\begin{equation*}
|\zeta_n|\sim n, \quad |\xi_n|\sim \frac{1}{\sqrt{n}}.
\end{equation*}
Therefore, given $\tilde\beta\in (0,\frac{1}{2})$, for all $n$ sufficiently large, 
\begin{equation}\label{xn}
|\zeta_n|^{\tilde\beta-1}<|\xi_n|<|\zeta_n|^{-\tilde\beta}.
\end{equation}
Moreover, since $\zeta_n=\zeta_0+n$, it follows that, given $\tilde R>0$ and  $\tilde \theta\in (0,\frac{\pi}{2})$, for all $n$ sufficiently large,
\begin{equation}\label{Hxn}
\zeta_n\in H(\tilde R, \tilde\theta).
\end{equation}

Note that $G(z,w)=Q(z,w)=(\psi(z,w), \sigma(z,w))$ for all $(z,w)\in B$. Hence, by Proposition \ref{local-inj}, there exist $\tilde\beta\in (0,\frac{1}{2})$, $\tilde\theta\in (0,\pi/2)$ and $\tilde R>1$ such that $\{(U,w)\in\C^2: U\in H(\tilde R, \tilde\theta), |U|^{\tilde\beta-1}<|w|<|U|^{-\tilde\beta}\}\subset G(B)$. Therefore, from \eqref{xn} and \eqref{Hxn}, it follows at once that  $H\times \C^\ast\subseteq G(\Omega_0)$, and, in fact, equality holds since $\Omega_0$ --- and hence $G(\Omega_0)$ --- is not simply connected. 
\end{proof}

We finally have all ingredients to prove the final result of this section.

\begin{proposition}\label{CxCstar}
$\Omega\simeq \C\times\C^\ast$.
\end{proposition}

\begin{proof}
Consider again $H:= g_1(B)$ and set $H_n:=H-n$. Since $\psi(B)\subset H$, we clearly have $\bigcup_{n\in\N} H_n=\C$. For each $n$, define $\varphi_n:g_1^{-1}(H_n)\to \C^2$ by
\[
\varphi_n(z,w):=G(F^n(z,w))-(n,0).
\]
Note that $g_1(F^n(z,w))=g_1(z,w)+n$, hence $F^n$ is a fiber preserving biholomorphism from $(g_1^{-1}(H_n))$ to $\Omega_0$. Therefore, by Proposition \ref{HxCstar}
\[
\varphi_n: g_1^{-1}(H_n)\to H_n\times \C^\ast
\]
is a fiber preserving biholomorphism. Moreover, for each $p\in \Omega$, if $F^n(p)\in \Omega_0$ we have
\[
G(F^{n+1}(p))=G(F(F^n(p)))=T(G(F^n(p))). 
\]
Now, take $\zeta\in H_n\cap H_{n+1}$ and let $w\in \C^\ast$. Note that $\zeta\mapsto \lambda e^{\frac{1}{2(\zeta+n)}}$ is a never vanishing holomorphic function on $H_n\cap H_{n+1}$. Hence, thanks to the previous equation, we have
\[
\varphi_n\circ \varphi_{n+1}^{-1}(\zeta, w)=(G\circ F^n)\circ (G\circ F^n)^{-1} T^{-1}(\zeta+n+1,w)-(n,0)=(\zeta, \lambda e^{\frac{1}{2(\zeta+n)} }w).
\]
This proves that $\Omega$ is a fiber bundle over $\C$ with fiber $\C^\ast$ and with transition functions $\zeta\mapsto \lambda e^{\frac{1}{2(\zeta+n)}}$ on $H_n\cap H_{n+1}$. 
In particular, $\Omega$ is a line bundle minus the zero section over $\C$. Since $H^1(\C, \mathcal O_\C^\ast)=0$, that is, all line bundles over $\C$ are (globally) holomorphically trivial, we obtain that $\Omega$ is biholomorphic to $\C\times\C^\ast$.
\end{proof}

\section{The global basin $\Omega$ and the Fatou component  containing $B$}\label{FGB}

Let $F$ be an automorphism of the form \eqref{automFp} as in the previous section, let $B$ be the local basin of attraction given by Theorem \ref{Thm:BZ} and $\Omega$ the associated global basin of attraction. Since $B$ is connected by Lemma \ref{Omega}, and $\{F^n\}$ converges to $(0,0)$ uniformly on  $B$, there exists an invariant Fatou component, which we denote by $V$, containing $B$, and we clearly have $\Omega\subseteq V$.
\smallskip

The aim of this section is to  characterize $\Omega$ in terms of orbits behavior, and to prove that  $\Omega=V$ under a generic condition on $\lambda$.

We use the same notations introduced in the previous sections. We start with the following corollary of Lemma~\ref{go-good-down}.

\begin{corollary}\label{Cor:how-Omega-is}
Let $F$ be an automorphism of $\C^2$ of the form \eqref{automFp}. Suppose that $\{(z_n,w_n):=F^n(z_0,w_0)\}$, the orbit under $F$ of a point $(z_0,w_0)$,  converges to $(0,0)$. Then $(z_0,w_0)\in \Omega$ if and only if $(z_n,w_n)$ is eventually contained in $W(\beta)$ for some---and hence any---$\beta\in (0,1/2)$ such that $\beta(l+1)>2$. 
\end{corollary}

\begin{proof}
If $(z_n,w_n)\in W(\beta)$  eventually for some $\beta\in (0,1/2)$ with $\beta(l+1)>2$ then, by Lemma \ref{go-good-down}, $(z_n,w_n)\in B$ eventually, and hence, $(z_0,w_0)\in \Omega$.  Conversely, if $(z_0,w_0)\in \Omega$, then $(z_n,w_n)\in W(\beta_0)$ eventually and $\beta_0(l+1)\geq 4$, and hence Lemma \ref{go-good-down} implies that $(z_n,w_n)\in W(\beta)$ eventually for any $\beta\in (0,1/2)$ such that $\beta(l+1)>2$.
\end{proof}

We can now prove the following characterization of $\Omega$.

\begin{theorem}\label{characterized Omega}
Let $F$ be an automorphism of $\C^2$ of the form \eqref{automFp}. Then,
\[
\Omega=\{(z,w)\in \C^2\setminus\{(0,0)\}: \lim_{n\to \infty}\|(z_n,w_n)\|=0, \quad |z_n|\sim |w_n|\},
\]
where $(z_n,w_n)=F^n(z,w)$.
\end{theorem}

\begin{proof} If $(z,w)\in \Omega$, then eventually $(z_n,w_n)\in W(\beta_0)$ and, hence, $|z_n|\sim |w_n|$ by Lemma~\ref{go-good-down}. On the other hand, if $(z_n,w_n)\to (0,0)$ and $|z_n|\sim |w_n|$, it follows that for every $\beta\in (0,1/2)$, $(z_n,w_n)\in W(\beta)$. Indeed, let $0<c_1<c_2$ be such that $c_1 |z_n|<|w_n|<c_2 |z_n|$ eventually. Let $\beta\in (0,1/2)$. Then for $n$ large,
\[
|z_n|^{\frac{1-\beta}{\beta}}<c_1 |z_n|<|w_n|,
\]
that is, $|z_n|<|u_n|^\beta$, and similarly it can be proved that $|w_n|<|u_n|^\beta$. Hence, by Corollary \ref{Cor:how-Omega-is}, $(z,w)\in \Omega$.
\end{proof}

In order to show that, under some generic arithmetic assumptions on $\lambda$, $\Omega$ coincides with the Fatou component which contains it, we need to prove some preliminary results.

\begin{lemma}\label{Change-of-coordinates}
Let $\chi$ be a germ of biholomorphism of $\C^2$ at $(0,0)$ given by
\[
\chi(z,w)=(z+A(z,w), w+B(z,w)),
\]
where $A$ and $B$ are germs of holomorphic functions at $(0,0)$ with $A(z,w)=O(\|(z,w)\|^h)$ and $B(z,w)=O(\|(z,w)\|^h)$ for some $h\geq 2$. Let $\beta\in (0,1/2)$. Assume that $\beta(h+1)>1$.
Then for any $\beta'\in (0,\beta)$ there exists $\epsilon>0$ such that for every $(z,w)\in W(\beta)$ with $\|(z,w)\|<\epsilon$ it holds $\chi(z,w)\in W(\beta')$.
\end{lemma}

\begin{proof}
Let us write $(\tilde{z}, \tilde{w})=\chi(z,w)$. Then we have $\tilde{z}=z+A(z,w)$ and $\tilde{w}=w+B(z,w)$.

Fix $r>0$, $\beta\in (0,1/2)$ such that $\beta(h+1)>1$, and $\beta'\in (0,\beta)$. By definition, for $\|(z,w)\|<r$, if $(z,w)\in W(\beta)$, then there exists a constant $C>0$ such that $|A(z,w)|\leq C |zw|^{\beta h}$ and $|B(z,w)|\leq C|zw|^{\beta h}$. Hence, for all $(z,w)\in W(\beta)$ with $\|(z,w)\|<r$,
\[
|\tilde{z}|\leq |z|+|A(z,w)|<|zw|^\beta+C |zw|^{\beta h}=|zw|^\beta(1+o(|zw|^{\beta(h-1)})), 
\]
and similarly, $|\tilde{w}|<|zw|^\beta(1+o(|zw|^{\beta(h-1)}))$. Therefore, since $\beta(h+1)>1$,
\begin{equation*}
\begin{split}
|\tilde z \tilde w|&\geq |zw|-|z||B|-|w||A|-|AB|\\&
\geq |zw|-2C|zw|^{\beta (h+1)}-C^2|zw|^{2 h\beta}\\
&= |zw|(1+o(|zw|^{\beta(h+1)-1})).
\end{split}
\end{equation*}
It thus follows that, for $(z,w)\in W(\beta)$ sufficiently close to $(0,0)$, we have
\[
|\tilde{z}|<|zw|^\beta(1+o(|zw|^{\beta(h-1)}))\leq |\tilde z \tilde w|^{\beta}\frac{1+o(|zw|^{\beta(h-1)})}{1+o(|zw|^{\beta(h+1)-1})}\leq |\tilde z \tilde w|^{\beta}(1+o(1))<|\tilde z \tilde w|^{\beta'}.
\]
A similar argument holding for $\tilde{w}$, the statement is proved.
\end{proof}

\begin{remark}
Note that the previous lemma does not hold without the hypothesis $\beta(h+1)>1$. Consider for instance the holomorphic map $\chi(z,w)=(z+w^2, w)$. Then the points of the form $(-w^2,w)$ belong to $W(\beta)$ for all $\beta<1/3$ but $\chi(-w^2,w)=(0,w)\not\in W(\beta')$ for any $\beta'\in (0,1/2)$.
\end{remark}

To state and prove Theorem \ref{Fatou-Omega} we also need one more assumption, namely an arithmetic condition on the eigenvalue $\lambda$.

Let $\lambda\in \C$ be such that $|\lambda|=1$. Recall that $\lambda$ is called {\sl Siegel}  if there exist $c>0$ and $N\in \N$ such that $|\lambda^k-1|\geq c k^{-N}$ for all $k\in \N$, $k\geq 1$ (such a condition holds for $\theta$ in a full Lebesgue measure subset of the unit circle, see, {\sl e.g.}, \cite{Po}). More generally, one says that a number $\lambda$ is {\sl Brjuno} if
\begin{equation}\label{eq:brjuno}
\sum_{k=0}^{+\infty}{\frac{1}{2^k}}\log{\frac{1}{\omega(2^{k+1})}}<+\infty\;,
\end{equation}
where $\omega(m) = \min_{2\le k\le m} |\lambda^k - \lambda|$ for any $m\ge 2$. Roughly speaking, 
the logarithm of a Brjuno number is badly approximated by rationals (see \cite{Brjuno} or \cite{Po}  for more details). Siegel numbers are examples of Brjuno numbers. 

\begin{lemma}\label{Brunocoord}
Let $F$ be given by \eqref{automFp-local}. If $\lambda$ is Brjuno, then there exists a germ of biholomorphism $\chi$ of $\C^2$ at $(0,0)$ of the form $\chi(z,w)=(z,w)+O(\|(z,w)\|^l)$, such that 
\begin{equation}\label{F-bruno}
\tilde F(\tilde z, \tilde w):=(\chi \circ F \circ \chi^{-1})(\tilde z,\tilde w)=(\lambda \tilde z  +  \tilde z\tilde w A(\tilde z, \tilde w), \overline{\lambda} \tilde w + \tilde z\tilde w B(\tilde z, \tilde w)),
\end{equation}
where $A, B$ are germs of holomorphic functions at $(0,0)$.
\end{lemma}

\begin{proof}
Thanks to the fact that $\lambda$ is Brjuno, the divisors $\lambda^k-\lambda$ and $\lambda^k-\overline{\lambda}$ are ``admissible'' in the sense of P\"oschel \cite{Po} for all $k\in \N$, $k\geq 2$. Hence, by  \cite[Theorem 1]{Po}, there exist $\delta>0$ and an injective holomorphic map $\varphi_1:\D_\delta \to \C^2$, where $\D_\delta:=\{\zeta\in \C: |\zeta|<\delta\}$, such that $\varphi_1(0)=(0,0)$, $\varphi_1'(0)=(1,0)$ and 
\begin{equation}\label{poschel1}
F(\varphi_1(\zeta))=\varphi_1(\lambda\zeta),
\end{equation}  
for all $\zeta\in \D_\delta$.  Since $F$ is tangent to $\{w=0\}$ up to order $l$, if follows from the proof of \cite[Theorem 1]{Po} that $\varphi_1$ can be chosen of the form $\varphi_1(\zeta)=(\zeta,0)+O(|\zeta|^l)$. In particular, up to shrinking $\delta$, we can write $\varphi_1(\D_\delta)$ implicitly as $w=\psi_1(z)$ for some holomorphic function $\psi_1$ defined on $\D_\delta$ and such that $\psi_1(\zeta)=O(|\zeta|^l)$.

Similarly, $\overline{\lambda}^k-\lambda$ and $\overline{\lambda}^k-\overline{\lambda}$ are admissible divisors in the sense of P\"oschel for all $k\in \N$, $k\geq 2$ and hence there exist $\delta'>0$ and a holomorphic function $\psi_2:\D_{\delta'}\to \C$ with $\psi_2(\zeta)=O(|\zeta|^l)$, such that $F$ leaves invariant the local curve $C:=\{(z,w): z=\psi_2(w)\}$ and the restriction of $F$ to $C$ is a $\overline{\lambda}$-rotation. 

We can therefore define $(\tilde z, \tilde w):=\chi(z,w)=(z-\psi_2(w), w-\psi_1(z))$. By construction, $\chi$ is a germ of biholomorphism at $(0,0)$ and $\chi(z,w)=(z,w)+O(\|(z,w)\|^l)$. Moreover, the conjugate germ $\tilde F(\tilde z, \tilde w):=(\chi \circ F \circ \chi^{-1})(\tilde z,\tilde w)$ satisfies our thesis. Indeed, $\tilde{z}=0$ corresponds to $z-\psi_2(w)=0$, and since $F$ leaves such a curve invariant and it is a $\overline{\lambda}$-rotation on it, it follows that $\tilde F(0, \tilde w)=(0,\overline{\lambda}\tilde w)$. A similar argument proves that $\tilde F( \tilde z, 0)=(\lambda \tilde z,0)$.
\end{proof}

The last ingredient in the proof of Theorem~\ref{Fatou-Omega} is the following  fact which can be easily proved via standard estimates:

\begin{lemma}\label{Lem:metric}
Let $\D^\ast=\{\zeta\in \C: 0<|\zeta|<1\}$. Let $k_{\D^\ast}$ denote the hyperbolic distance in $\D^\ast$. Let 
\[
g(\zeta,\xi):=2\pi \max\left\{-\frac{1}{\log |\zeta|},-\frac{1}{\log |\xi|}\right\}.
\]
Then for all $\zeta, \xi\in \D^\ast$ it holds
\[
\left|\log\frac{\log|\zeta|}{\log|\xi|} \right|- g(\zeta,\xi)\leq k_{\D^\ast}(\zeta, \xi)\leq \left|\log\frac{\log|\zeta|}{\log|\xi|} \right|+g(\zeta,\xi).
\]
\end{lemma}

Now we are in a good shape to state and prove the main result of this section:

\begin{theorem}\label{Fatou-Omega}
Let $F$ be an automorphism of $\C^2$ of the form \eqref{automFp}. If $\lambda$ is Brjuno, then $\Omega=V$.
\end{theorem}

\begin{proof}
Assume by contradiction that the statement is not true. Hence, there exists $q_0\in V\setminus \Omega$. Let $p_0\in \Omega$, and let $Z$ be an open connected set containing $p_0$ and $q_0$ and such that $\overline{Z}\subset V$. 

By Lemma \ref{Brunocoord}, since $\lambda$ is Brjuno, there exists an open neighborhood $U$ of $(0,0)$ and a biholomorphism $\chi:U \to \chi(U)$, such that \eqref{F-bruno} holds for all $(\tilde z,\tilde w)\in \chi(U)$. Up to rescaling, we can assume that 
\[
\B^2:=\{(\tilde z,\tilde w)\in \C^2: |\tilde z|^2+|\tilde w|^2<1\}\subset \chi(U).
\]
Since $\{F^n\}$ converges uniformly to $(0,0)$ on $\overline{Z}$, up to replacing $F$ with $F^m$ for some fixed $m\in \N$, we may assume that $Q:=\cup_{n\in \N}F^n(Z)$ satisfies $\tilde Q:=\chi(Q)\subset \B^2$.

The axes $\tilde z$ and $\tilde w$ are $\tilde F$-invariant and $\tilde F$ is a rotation once restricted to the axes, therefore 
\[
\tilde Q\subset \B^2_\ast:=\B^2\setminus(\{\tilde z=0\}\cup \{\tilde w=0\}).
\]

Given a complex manifold $M$, we denote by $k_M$ its Kobayashi distance.  By construction, for every $\delta>0$,  one can find $p\in Z\cap \Omega$ and $q\in Z\cap (V\setminus \Omega)$ such that $k_Q(p,q)\leq k_Z(p,q)<\delta$. Let $\tilde p:=\chi(p)$ and $\tilde q:=\chi(q)$. Hence, $k_{\tilde Q}(\tilde p, \tilde q)<\delta$. Thus, since $\tilde F(\tilde Q)\subset \tilde Q$ by construction, and $\tilde{Q}\subset \B^2_\ast$, it follows that for all $n\in \N$,
\begin{equation}\label{delta-sub}
k_{\B^2_\ast}(\tilde F^n(\tilde p), \tilde F^n(\tilde q))\leq  k_{\tilde Q}(\tilde F^n(\tilde p), \tilde F^n(\tilde q))<\delta.
\end{equation}
Now, since $q\not\in \Omega$, by Lemma \ref{go-good-down}, there is no $\beta\in (0,1/2)$ with $\beta(l+1)>2$ such that $\{F^n(q)\}\subset W(\beta)$ eventually. We claim that the same happens to $\{\tilde F^n (\tilde q)\}$. Indeed, if there existed $\beta\in (0,1/2)$ with $\beta(l+1)>2$ such that $\{\tilde F^n (\tilde q)\}\subset W(\beta)$ eventually, taking $\beta'\in (0,\beta)$ so that $\beta'(l+1)>2$, Lemma \ref{Change-of-coordinates} applied to $\chi^{-1}(\tilde z, \tilde w)=(\tilde z, \tilde w)+O(\|(\tilde z, \tilde w)\|^l)$ would imply that $\{F^n(q)\}\subset W(\beta')$ eventually, contradicting our assumption.

Therefore, fixing $\beta\in (0,1/2)$ with $\beta(l+1)>2$, we can assume, without loss of generality, that there exists an increasing subsequence $\{n_k\}\subset \N$ tending to $\infty$ such that, setting $(\tilde z_n(\tilde q),\tilde w_n(\tilde q)):=\tilde F^n(\tilde q)$, for all $n_k$ it holds $|\tilde z_{n_k}(\tilde q)|\geq |\tilde z_{n_k}(\tilde q)\tilde w_{n_k}(\tilde q)|^{\beta}$, that is
\begin{equation}\label{Eq:bad-go}
|\tilde w_{n_k}(\tilde q)|\leq |\tilde z_{n_k}(\tilde q)|^{\frac{1-\beta}{\beta}}.
\end{equation}
On the other hand, by Lemma \ref{go-good-down},  $\{F^n(p)\}\subset W(\beta)$ eventually for all $\beta\in (0,1/2)$ such that $\beta(l+1)>2$. Hence, by Lemma \ref{Change-of-coordinates}, it follows that $\{\tilde F^n(\tilde p)\}\subset W(\beta')$ eventually for all $\beta'\in (0,\beta)$. Since this holds for all $\beta\in (0,1/2)$ such that $\beta(l+1)>2$, we obtain that $\{\tilde F^n(\tilde p)\}\subset W(\beta)$ eventually. Therefore, again by Lemma \ref{go-good-down}, there exist $0<c<C$ and $\tilde n>0$ such that for all $n\ge \tilde n$ 
\begin{equation}\label{Eq:good-go}
c|\tilde z_n(\tilde p)|\leq |\tilde w_n(\tilde p)|\leq C |\tilde z_n(\tilde p)|.
\end{equation} 

Consider the holomorphic projections $\pi_1:\B^2_\ast\to \D^\ast$ given by $\pi_1(\tilde z, \tilde w)=\tilde z$, and $\pi_2:\B^2_\ast\to \D^\ast$ given by $\pi_2(\tilde z, \tilde w)=\tilde w$. By the properties of the Kobayashi distance, $k_{\D^\ast}(\pi_j(A), \pi_j(B))\leq k_{\B^2_\ast}(A, B)$ for every $A, B\in \B^2_\ast$. Hence, by \eqref{delta-sub}, for all $n_k$,
\begin{equation}\label{eq:stima delta}
k_{\D^\ast}(\tilde z_{n_k}(\tilde p), \tilde z_{n_k}(\tilde q))<\delta, \quad k_{\D^\ast}(\tilde w_{n_k}(\tilde p), \tilde w_{n_k}(\tilde q))<\delta.
\end{equation}
Thanks to \eqref{Eq:good-go} and Lemma \ref{Lem:metric}, since the orbit of $\tilde p$ converges to the origin, there exists $k_0\in \N$ such that for all $n_k\geq k_0$,
\[
k_{\D^\ast}(\tilde z_{n_k}(\tilde p), \tilde w_{n_k}(\tilde p))\leq \left|\log\frac{\log|\tilde z_{n_k}(\tilde p)|}{\log|\tilde w_{n_k}(\tilde p)|} \right|+g(\tilde z_{n_k}(\tilde p),\tilde w_{n_k}(\tilde p))<\delta.
\]
Hence, by \eqref{eq:stima delta} and the triangle inequality, for all $n_k\geq k_0$,  
\begin{equation}\label{eq:stimona1}
k_{\D^\ast}(\tilde z_{n_k}(\tilde q), \tilde w_{n_k}(\tilde p))\leq k_{\D^\ast}(\tilde z_{n_k}(\tilde q), \tilde z_{n_k}(\tilde p))+k_{\D^\ast}(\tilde z_{n_k}(\tilde p), \tilde w_{n_k}(\tilde p))<2\delta.
\end{equation}
On the other hand,  let $k_1\in \N$ be such that, for all $n_k\geq k_1$,
\[
g(\tilde z_{n_k}(\tilde q),\tilde w_{n_k}(\tilde q))<\delta,
\]
where $g$ is the function defined in Lemma \ref{Lem:metric}. By the same lemma and \eqref{Eq:bad-go}
\begin{equation}\label{eq:stimona2}
\begin{split}
k_{\D^\ast}(\tilde z_{n_k}(\tilde q), \tilde w_{n_k}(\tilde q))
&\geq \left|\log\frac{\log|\tilde z_{n_k}(\tilde q)|}{\log|\tilde w_{n_k}(\tilde q)|} \right|-g(\tilde z_{n_k}(\tilde q),\tilde w_{n_k}(\tilde q))\\ 
&\geq \log\left(\frac{\log|\tilde z_{n_k}(\tilde q)|^{\frac{1-\beta}{\beta}}}{\log|\tilde z_{n_k}(\tilde q)|} \right)-\delta= \log \frac{1-\beta}{\beta} -\delta.
\end{split}
\end{equation}
The triangle inequality, together with \eqref{eq:stimona1} and \eqref{eq:stimona2} yield that for $n_k\geq \max\{k_0, k_1\}$
\begin{equation*}
\begin{split}
k_{\D^\ast}(\tilde w_{n_k}(\tilde p), \tilde w_{n_k}(\tilde q))&\geq k_{\D^\ast}(\tilde z_{n_k}(\tilde q), \tilde w_{n_k}(\tilde q))-k_{\D^\ast}(\tilde z_{n_k}(\tilde q), \tilde w_{n_k}(\tilde p))\\&\geq \log \frac{1-\beta}{\beta} -3\delta.
\end{split}
\end{equation*}
Therefore, by \eqref{eq:stima delta}, 
\[
4\delta\geq \log \frac{1-\beta}{\beta},
\]
giving a contradiction since $\frac{1-\beta}{\beta}>1$ is fixed and $\delta>0$ is arbitrary.
\end{proof}

\section{The proof of Theorem \ref{main} for~$k=2$}

Let $F$ be an automorphism of the form \eqref{automFp}, and assume that $\lambda$ is Brjuno. By Theorem~\ref{Fatou-Omega},  $\Omega$ is an invariant attracting Fatou component at $(0,0)$ and $\Omega$ is biholomorphic to $\C\times \C^\ast$ by Proposition~\ref{CxCstar}.

\section{The case~$k\ge 3$}

In the general case, $k\ge 3$, we start with a germ of biholomorphism of $\C^k$ at the origin of the form
\begin{equation}\label{form-intro-gen}
F_N(z_1,\dots,z_k)=\left(\lambda_1 z_1\left(1 - \frac{z_1\cdots z_k}{k}\right), \dots,
\lambda_k z_k\left(1 - \frac{z_1\cdots z_k}{k}\right)\right),
\end{equation}
where 
\begin{enumerate}
\item each $\lambda_j\in \C$, $|\lambda_j|=1$, is not a root of unity for $j=1,\dots, k$,
\item the $k$-tuple $(\lambda_1,\dots, \lambda_k)$ is {\em one-resonant with index of resonance $(1,\dots,1)\in\N^k$} in the sense of \cite[Definition 2.3]{BZ}, that is all the resonances $\lambda_j - \lambda_1^{m_1}\cdots\lambda_k^{m_k}=0$, for $j=1,\dots, k$, are precisely of the form $\lambda_j = \lambda_j\cdot\left(\lambda_1\cdots\lambda_n\right)^{k}$ for some $k\ge 1$,
\item the $k$-tuple $(\lambda_1,\dots, \lambda_k)$ is {\em admissible} in the sense of P\"oschel (see \cite{Po}), that is we have
\[
\sum_{n=0}^{+\infty}{\frac{1}{2^n}}\log{\frac{1}{\omega_j(2^{n+1})}}<+\infty\;,~\hbox{for}~j=1,
\dots, k
\]
where $\omega_j(m) = \min_{2\le h\le m} \min_{1\le i\le k}|\lambda_j^h - \lambda_i|$ for any $m\ge 2$.

\end{enumerate}

Thanks to a result of  B. J.  Weickert \cite{W1} and F. Forstneri\v{c} \cite{F}, for any large $l\in \N$ there exists an automorphism $F$ of $\C^k$ such that 
\begin{equation}\label{Eq-motiv}
F(z_1,\dots, z_k)-F_N(z_1,\dots, z_k)=O(\|(z_1,\dots, z_k)\|^l).
\end{equation}
Moreover, thanks to \cite[Theorem 1.1]{BZ}, given $\beta\in (0,\frac{1}{k})$ and $l\in\N$, $l\ge 4$ such that $\beta(l+1)\ge 4$, for every $\theta\in (0,\frac{\pi}{2})$, there is $R>0$ such that the open set 
\[
B:=\{(z_1,\dots, z_k)\in \C^k: u:= z_1\cdots z_k\in S(R,\theta), |z_j|<|u|^\beta~\hbox{for}~j=1,\dots, k\},
\]
is non-empty, forward invariant under $F$, the origin is on the boundary of $B$ and we have $\lim_{n\to \infty}F^n(p)=0$ for all $p\in B$, uniformly on compacta. Arguing as in Lemma~\ref{go-good-down} we obtain that for each $p\in B$, we have that $\lim_{n\to \infty} nu_n
=1$ and $|\pi_j(F^n(p))|\sim n^{-1/k}$, for $j=1,\dots, k$, where $\pi_j$ is the projection on the $j$th coordinate. Moreover, the analogue of the statement of Proposition~\ref{BRZ} holds for $k\ge 3$ (see also \cite{BRZ}) allowing to define a local Fatou coordinate $\psi \colon B\to \C$ such that $\psi\circ F = \psi+1$ with the required properties.

Now we need $k-1$ other local coordinates $\sigma_2, \dots, \sigma_{k}$. For $2\le j\le k$, $\sigma_j\colon B\to \C$ is defined as the uniform limit on compacta of the sequence $\{\sigma_{j,n}\}_n$ where 
\begin{equation*}
\sigma_{j,n}(z_1,\dots, z_k):= (\lambda_j\dots\lambda_k)^{-n} \Pi_j(F^n(z_1,\dots, z_k)) \exp\left({\frac{k-j+1}{k}\sum_{m=0}^{n-1} \frac{1}{\psi(z_1,\dots, z_k)+m}}\right),
\end{equation*}
and $\Pi_j\colon \C^k\to\C$ is defined as $\Pi_j(z_1,\dots, z_k) := z_j\cdots z_k$. The map $\sigma_j$ satisfies the functional equation
$$
\sigma_j\circ F = {\lambda_j\cdots\lambda_k} e^{-\frac{k-j+1}{k\psi}}\sigma_j.
$$

Let $\Omega:=\cup_{n\ge 0} F^{-n}(B)$. Arguing like in dimension $2$, one can prove that $H^{k-1}(\Omega,\C)\ne 0$. Using the functional equation we can extend $\psi$ to a map $g_1\colon \Omega \to \C$. Moreover, set $H:=g_1(B)$ and $\Omega_0:=g_1^{-1}(H)$. For $j=2,\dots, k$, we can extend $\sigma_j$ to $\Omega_0$ by setting, for any $p\in \Omega_0$,
$$
g_j(p) = (\lambda_j\cdots\lambda_k)^n \exp{\left(-\frac{k-j+1}{k}\sum_{m=0}^{n-1}\frac{1}{g_1(p)+j}\right)}\sigma_j(F^n(p))
$$ 
where $n\in\N$ is so that $F^n(p)\in B$.
As in dimension $2$, the map $\Omega_0\ni p\mapsto G(p):=(g_1(p),\dots, g_k(p))\in H\times \C^{k-1}$ is univalent with image $H\times (\C^*)^{k-1}$. In fact, we can use coordinates 
$$
(u, y_2, \dots, y_k) := (z_1\cdots z_k,z_2\cdots z_k, \dots, z_k),
$$
in $B$ so that we have 
$$
B = \{u\in S(R,\theta), |u|^{1-k\beta}<|y_k|<|u|^\beta,~|u|^{1-j\beta}<|y_j|<|u|^\beta |y_{j+1}|~\hbox{for}~j=2,\dots, k-1\}.
$$
Following the proof of Proposition~\ref{HxCstar}, since, for $p\in \Omega_0$, $\lim_{n\to \infty} nu_n=1$ and $|\Pi_j(F^n(p))|\sim n^{-(k-j+1)/k}$ for $j=2,\dots, k$ one can see that for any $a\in H$ and $b_k\in\C^*$ there is a point $p\in \Omega_0$ such that $g_1(p)=a$ and $g_k(p)=b_k$. Now fix $a\in H$ and $b_k\in\C^*$. 
Using 
$$
|u|^{1-(k-2)\beta}<|y_{k-1}|<|u|^\beta |y_{k}|
$$
one sees that $\C^*\subseteq g_{k-1}(g_1^{-1}(a)\cap g_k^{-1}(b_k))$, and so on for every $j=2,\dots, k-2$. Therefore $G(\Omega_0)=H\times (\C^*)^{k-1}$, and as in Proposition~\ref{CxCstar} we see that $g_1\colon \Omega\to\C$ is a holomorphic fiber bundle map with fiber $(\C^*)^{k-1}$. Since the transition functions belong to ${\rm GL}_n(\C)$, by \cite[Corollary~8.3.3]{Franz} we obtain that $\Omega$ is biholomorphic to $\C\times(\C^*)^{k-1}$.

Finally, assuming the $k$-tuple $(\lambda_1,\dots, \lambda_k)$ to be admissible in the sense of P\"oschel \cite{Po}, we can locally choose coordinates as in Lemma~\ref{Brunocoord} so that the Fatou component $V$ containing $\Omega$ cannot intersect the coordinate axes in a small neighborhood of the origin. Hence using the estimates for the Kobayashi distance as done in Theorem~\ref{Fatou-Omega}, one can show that $V=\Omega$.
 
\bibliographystyle{amsplain}

\end{document}